\documentclass{amsart}
\usepackage{amssymb, amsthm, mathtools,tikz-cd}
\usepackage{comment}
\usepackage[shortlabels]{enumitem}
\usetikzlibrary{calc, intersections}
\usepackage{multicol}

\newtheorem{thm'}{Theorem} 
\theoremstyle{definition}
\newtheorem*{conv*}{Indexing Convention}

\newtheorem{thm}{Theorem}[section]
\theoremstyle{definition}
\newtheorem{defn}[thm]{Definition}
\theoremstyle{definition}
\newtheorem{rmk}[thm]{Remark}
\theoremstyle{definition}
\newtheorem{ex}[thm]{Example}

\newtheorem{lem}[thm]{Lemma}
\newtheorem{prop}[thm]{Proposition}
\newtheorem{cor}[thm]{Corollary}
\newtheoremstyle{case}{}{}{}{}{\itshape}{\textup{:}}{ }{}
\theoremstyle{case}
\newtheorem{case}{Case}

\newcommand{\C}{\mathcal{C}}
\newcommand{\D}{\mathcal{D}}

\renewcommand{\H}{\mathcal{H}}
\renewcommand{\L}{\mathcal{L}}

\newcommand{\CC}{\mathbb{C}}

\newcommand{\PP}{\mathbb{P}}
\newcommand{\RR}{\mathbb{R}}
\newcommand{\QQ}{\mathbb{Q}}
\newcommand{\ZZ}{\mathbb{Z}}

\numberwithin{equation}{section}

\title[A Topological Characterization of the Intersection Complex]{A Topological Characterization of the Middle Perversity Intersection Complex for Arbitrary Complex Algebraic Varieties}
\author{Ben Wu}

\begin{document}

\begin{abstract}
For an arbitrary complex algebraic variety which is not necessarily pure dimensional, the intersection complex can be defined as the direct sum of the Deligne-Goresky-MacPherson intersection complexes of each irreducible component. We give two axiomatic topological characterizations of the middle perversity direct sum intersection complex, one stratification dependent and the other stratification independent. To accomplish this, we show that this direct sum intersection complex can be constructed using Deligne's construction in the more general context of topologically stratified spaces. A consequence of these characterizations is the invariance of this direct sum intersection complex under homeomorphisms. 
\end{abstract}

\maketitle

\tableofcontents

\section{Introduction}
In \cite{GM1}, Goresky and MacPherson introduce the intersection (co)homology groups for a topological pseudomanifold. In \cite{GM2}, Goresky and MacPherson construct a complex of sheaves whose (hyper)cohomology gives the intersection homology groups. This complex of sheaves is called the \textit{(Deligne-Goresky-MacPherson) intersection complex} and the construction is referred to as \textit{Deligne's construction} (the indexing convention used for intersection complexes is discussed at the end of the introduction). They show that the intersection complex is uniquely characterized (up to canonical isomorphism) by certain axioms. A consequence of this characterization is that the intersection complex, and hence the intersection homology, is invariant under homeomorphisms. Irreducible, or even pure dimensional, (complex algebraic) varieties can be viewed as topological pseudomanifolds and intersection homology is a useful tool for understanding their topology; see \cite{dCL}. Arbitrary varieties, however, cannot be viewed as topological pseudomanifolds because their irreducible components may have differing dimensions. Instead, they must be viewed as topologically stratified spaces; see $\S \ref{top strat space}$.

For arbitrary varieties, there is still a natural candidate for the intersection complex. In \cite{dC1}, de Cataldo defines the middle perversity intersection complex of a variety as a direct sum of the middle perversity Deligne-Goresky-MacPherson intersection complexes of each irreducible component. He then observes that this complex satisfies virtually all of the properties of the usual intersection complex for irreducible varieties, e.g. Poincar\'e duality, existence of mixed and pure Hodge structures, Lefschetz theorems, etc. In \cite{dCM}, de Cataldo and Maulik prove the homeomorphism invariance of the intersection complex as a lemma and use it to prove that the perverse Leray filtration for the Hitchin morphism is independent of the complex structure of the curve. An axiomatic characterization of the intersection complex, analogous to the one given by Goresky and MacPherson for pseudomanifolds, is desirable because it gives a topological criterion for determining which complexes can be the intersection complexes. Example \ref{ax 2 fails} in $\S 4$ shows that although each summand of the intersection complex is characterized by the axioms proposed by Goresky and MacPherson, it is not so clear which axioms characterize the direct sum.

The main goal of this paper is to give an axiomatic topological characterization of the middle perversity intersection complex of an arbitrary complex algebraic variety which is not necessarily pure dimensional. For those wondering why we only consider the middle perversity, see Remark \ref{arb perv rmk}. Although we will only work with complex algebraic varieties in this paper, our results hold for any topologically stratified space with only even dimensional strata (see Remark \ref{alg var vs strat space}). In particular, they also hold for complex analytic spaces. We summarize our approach below.

Let $X$ be a complex algebraic variety of complex dimension $n$, with stratification 
$$\mathfrak{X} : X= X_n \supseteq X_{n-1} \supseteq \cdots \supseteq X_{-1}= \varnothing$$
by closed subvarieties, so that all strata contained in $X_k - X_{k-1}$ are of pure complex dimension $k$ and $X$ (in the classical topology) has the structure of a topologically stratified space (e.g. $\mathfrak{X}$ induced by a Whitney stratification).  In section \ref{construction}, we show that the stratification induces an open dense subset $U \subseteq X$ such that each 
\begin{enumerate}
\item each point of $U$ admits a neighborhood homeomorphic to $\CC^m$ for some $1 \leq m \leq n$, i.e. $U = \bigsqcup_{m=1}^n U^m$ where $U^m$ is a topological manifold of complex dimension $m$,
\item $\overline{U^m}-U^m$ has complex dimension $\leq m-1$, where $\overline{U^m}$ is the closure of $U^m$ in $X$.
\end{enumerate}
We then construct a complex $IC(\mathfrak{X},\mathcal{L})$ of sheaves on $X$ using Deligne's construction with respect to: the (lower) middle perversity, the stratification $\mathfrak{X}$, and any local system $\mathcal{L}$ on the induced open dense subset $U \subseteq X$. Proposition \ref{IC is direct sum} shows that we can interpret this complex as a direct sum of Deligne-Goresky-MacPherson intersection complexes. In particular, for possibly reducible varieties, the complex $IC(\mathfrak{X},\mathcal{L})$ is the intersection complex defined by de Cataldo in \cite{dC1}. When the variety is pure dimensional, Deligne's construction begins by using the stratification to induce a filtration by open sets. A key ingredient in our construction of $IC(\mathfrak{X},\mathcal{L})$ is a new way of using the stratification to induce a filtration of a not necessarily pure dimensional complex algebraic variety by open sets.  Example \ref{open filt ex} shows that this procedure is more subtle than one might initially expect.  A priori the complex $IC(\mathfrak{X},\mathcal{L})$ depends on the stratification $\mathfrak{X}$ and the local system $\mathcal{L}$.  Our main result is the following:

\begin{thm'}[\S \ref{top indep}]\label{main thm 1}
Let $X$ be a complex algebraic variety of complex dimension $n$ which is not necessarily pure dimensional. Let $U$ be an open dense subset of $X$ satisfying $(1)$ and $(2)$ above. 
Let $\L^m$ be a local system on $U^m$ and set $\L = \bigoplus_{m=1}^n \L^m$ (extend each $\L^m$ on $U^m$ to $U$ by zero). Then there exists a unique (up to canonical isomorphism) complex $IC(X, \mathcal{L})$ satisfying: 
\begin{enumerate}[(a)]
\item  (Normalization) There exists an open dense subset $V$ of $X$ such that $V = \bigsqcup_{m=1}^n V^m$ where $V^m$ is a topological manifold of complex dimension $m$, dim$_\CC(\overline{V^m} - V^m) \leq m-1$, and $IC(X,\L)|_{V^m} \simeq \L'^m[m]$ where $\L'$ is the unique extension of $\L^m|_{U^m \cap V^m}$ to $V^m$ (see Remark \ref{local sys rmk} for more details on $\L'$).
\item (Pure Dimensional Support) For $1 \leq m \leq n$, if $a > -m$, $$\text{dim}_\CC\{x \in \overline{V^m} \ | \ \H^a(i_x^*S) \neq 0\} < -a.$$ 
\item (Pure Dimensional Cosupport) For $1 \leq m \leq n$, if $a < m$, $$\text{dim}_\CC\{x \in \overline{V^m} \ | \ \H^a(i_x^!S) \neq 0\} < a.$$ 
\end{enumerate}
where $i_x: \{x\} \to X$ is the inclusion. 

In particular, the complex $IC(\mathfrak{X},\mathcal{L})$ is independent of the stratification and the complex $IC(X,\mathcal{L})$ is invariant under homeomorphisms.
\end{thm'}

\setcounter{thm'}{0}

This theorem gives another proof of the homeomorphism invariance of the intersection complex proved by de Cataldo and Maulik in \cite{dCM} for possibly reducible varieties. We prove our main theorem by giving two characterizations of the complex $IC(\mathfrak{X},\mathcal{L})$ in Section \ref{axiomatic char}. These characterizations are analogous to the stratification dependent characterization, [AX1], and the stratification independent characterization, [AX2], of the intersection complex of a topological pseudomanifold given by Goresky and MacPherson in \cite{GM2}. To emphasize the analogy with the axioms proposed by Goresky and MacPherson, we will denote our sets of axioms by [AX1$'$] and [AX2$'$]. More precisely, we give a stratification dependent collection of axioms, [AX1$'$], and prove that $IC(\mathfrak{X},\mathcal{L})$ is the unique complex (up to canonical isomorphism) satisfying axioms [AX1$'$]; see Definition \ref{AX$1'$} and Theorem \ref{main thm 2}. We discuss the differences between axioms [AX1$'$] and axioms [AX1] in Remark \ref{ax1' vs ax1}. We then give a stratification independent collection of axioms, [AX2$'$], and prove that axioms [AX2$'$] are equivalent to axioms [AX1$'$]; see Definition \ref{AX$2'$} and Proposition \ref{main prop}. We discuss the differences between axioms [AX2$'$] and axioms [AX2] in Remark \ref{ax2' vs ax2}. In Section \ref{top indep}, we finish the proof by giving a way to compare objects in $D^b_c(X)$ with respect to two different stratifications which may not have a common refinement.

\begin{conv*}\label{indexing conv}
Let $X$ be a complex algebraic variety of pure complex dimension $n$. Let $U \subseteq X$ be an open dense subset which is a topological manifold of dimension $n$ and let $\L$ be a local system on $U$. We require that the middle perversity Deligne-Goresky-MacPherson intersection complex $IC(X,\L)$ satisfies $IC(X, \L)|_U \simeq \L[n]$. With this convention, $IC(X,\L)$ is perverse.\\
\indent Let $X$ be a complex algebraic variety of complex dimension $n$ which is not necessarily pure dimensional. Let $U$ be an open dense subset of $X$ such that $U = \bigsqcup_{m=1}^n U^m$ where each $U^m$ is a topological manifold of dimension $n$ and dim$_\CC(\overline{U^m}-U^m) \leq m-1$. Let $X^m = \overline{U^m}$ be the closure of $U^m$ in $X$. In Corollary \ref{irred comp cor}, we show that $X^m$ can be interpreted as the union of all irreducible $m$-dimensional components of $X$. Let $\L^m$ be a local system on $U^m$ and set $\L = \bigoplus_{m=1}^n \L^m$ where each $\L^m$ is extended to $U$ by zero. The intersection complex of $X$ is defined to be $IC(X, \L) = \bigoplus_{m=1}^n IC(X^m,\L^m)$ where each $IC(X^m,\L^m)$ is normalized as above. In particular, $IC(X, \L)|_U \simeq \bigoplus_{m=1}^n \L^m[m]$.
Although this indexing convention seems more cumbersome than the Borel convention where the local systems are not shifted, it is more convenient to use when construct the complex $IC(\mathfrak{X},\L)$ (see Remark \ref{indexing conv rmk} for a more detailed discussion).
\end{conv*}

\subsection{Acknowledgements}
I would like to thank my advisor Mark de Cataldo for suggesting this problem and for the many useful discussions. I would also like to thank J\"org Sch\"urmann, Michael Albanese and Lisa Marquand for their comments and suggestions. I would finally like to thank the anonymous referee whose comments and suggestions helped improved this paper.

\section{Preliminaries}
We begin by fixing some terminology and notation. Given a set $A$ and a subset $B \subseteq A$, we denote by $B^c$ the set complement of $B$. The word \textit{variety} means a separated scheme of finite type over the complex numbers $\CC$. We endow varieties with the classical topology. In this case, Whitney showed that varieties admit the structure of Whitney stratified spaces \cite{W}. Verdier then showed that there exists a Whitney stratification such that each strata is complex algebraic \cite{V}. Finally, Teissier showed that varieties admit a canonical Whitney stratification for which the strata are algebraic \cite{Te}. We work with a fixed regular Noetherian ring $R$ with finite Krull dimension. We shall mainly be concerned with the cases that $R = \ZZ, \QQ$, or $\CC$. The word \textit{sheaf} means a sheaf of $R$-modules. The constant sheaf on a topological space $X$ is denoted by $R_X$. The word \textit{complex} means a complex of sheaves of $R$-modules. Let $Sh(X)$ denote the abelian category of sheaves on $X$, and $D^b(X)$ denote the bounded derived category of the abelian category $Sh(X)$.

\subsection{Topologically Stratified Spaces} \label{top strat space}
We begin by recalling the basic definitions associated with topologically stratified spaces given in \cite{GM2}. A more detailed discussion can be found in \cite[Ch. 2]{F}.
\begin{defn}\label{top strat space def}
The definition of a topological stratified space is inductive. A $0$-dimensional topologically stratified Hausdorff space is a countable collection of points with the discrete topology.  An $n$-dimensional \textit{topological stratification} of a paracompact Hausdorff space $X$ is a finite filtration $\mathfrak{X}$ by closed subsets 
\begin{equation} \label{strat eq}
\mathfrak{X}: X = X_n \supseteq X_{n-1} \supseteq \cdots \supseteq X_0 \supseteq X_{-1}=\varnothing
\end{equation}
such that for each point $p \in X_k - X_{k-1}$, there exists a neighborhood $N$ of $p$, a compact Hausdorff space $L$ with an $(n-k-1)$-dimensional topological stratification 
\begin{equation}
L = L_{n-k-1} \supseteq \cdots \supseteq L_0 \supseteq L_{-1}= \varnothing,
\end{equation}
and a homeomorphism
\begin{equation} \label{local struc eq}
\phi : \RR^k \times \text{cone}^o(L) \to N
\end{equation}
which takes each $\RR^k \times \text{cone}^o(L_j)$ homeomorphically to $N \cap X_{k+j+1}$. Here, cone$^o(L)$ denotes the open cone $L \times [0,1) / \sim$ where $(l,0) \sim (l',0)$ for all $l,l' \in L$. We use the convention that cone$^o(\varnothing)$ is a point. 
We often refer to $N$ as a \textit{distinguished neighborhood} and $\mathfrak{X}$ as a stratification. In Remark \ref{projection rmk}, we emphasize some important structure of distinguished neighborhoods. To maintain simplicity in our formulas later on, we will make the assumption that stratified spaces do not contain any open $0$-dimensional strata, i.e. isolated points.
\end{defn}

If $X_k - X_{k-1}$ is nonempty, then for any $p \in X_k - X_{k-1}$, any distinguished neighborhood $N$ gives a homeomorphism $N \cap X_{k} \simeq \RR^k \times \text{cone}^o(L_{-1}) \simeq \RR^k$. By shrinking $N$ we can take $N \subseteq X_{k-1}^c$. Thus, if $X_k- X_{k-1}$ is nonempty, it is a $k$-dimensional topological manifold. The connected components of $X_k- X_{k-1}$ are called the \textit{$k$-dimensional strata} of $X$.

A consequence of the definition is that stratified spaces satisfy the \textit{axiom of the frontier}, i.e. the closure of any stratum is a union of lower dimensional strata. We refer the reader to \cite[\S 2.2-\S 2.3]{F} for proofs. 

\begin{rmk}\label{projection rmk}
Let $X$ be a stratified space with stratification $\mathfrak{X}$ and $N \simeq \RR^k \times cone^o(L)$ be a distinguished neighborhood of $x \in X_k-X_{k-1}$. Let $\pi : N \to cone^o(L)$ denote the natural projection map. There is a natural stratification on $\RR^k \times cone^o(L)$ given by setting $(\RR^k \times cone^o(L))_j \coloneqq \RR^k \times cone^o(L_j)$. Since $\RR^k \times \text{cone}^o(L_j)$ homeomorphic to $N \cap X_{k+j+1}$, the natural stratification on $\RR^k \times cone^o(L)$ is the same as the stratification on $N$ induced by $\mathfrak{X}$. In particular, if $S$ is a stratum of $X$, then $S \cap N$ is a union of strata of the form $\RR^k \times cone^o(T)$ where $T$ is a stratum of $L$. It follows that $\pi^{-1}(\pi(S \cap N)) = S \cap N$.
\end{rmk}

\begin{rmk} \label{even dim strat rmk}
A Whitney stratification on a complex algebraic variety $X$ induces a topological stratification. Thus, we can view $X$ as a topologically stratified space with only even dimensional strata. We will denote this stratification by
\begin{equation*}
\mathfrak{X} : X = X_n \supseteq X_{n-1} \supseteq \cdots \supseteq X_0 \supseteq X_{-1} = \varnothing
\end{equation*}
where $X_k - X_{k-1}$ consists of complex $k$-dimensional strata. The strata can be taken to be complex algebraic, but we will not need this fact. A stratification of a complex algebraic variety will always mean stratification in the above sense. 
\end{rmk}

\begin{defn}
A topologically stratified space $X$ is \textit{purely $n$-dimensional} if $X_n-X_{n-1}$ is dense in $X$. A topologically stratified space is purely $n$-dimensional if and only if every open set has topological dimension $n$ in the sense of Hurewicz and Wallman described in \cite{HW}. An \textit{$n$-dimensional topological pseudomanifold} is a purely $n$-dimensional topologically stratified space which admits a stratification $\mathfrak{X}$ such that $X_{n-1} = X_{n-2}$. 
\end{defn}

\begin{defn}
Let $X$ and $Y$ be stratified spaces. A continuous map $f: X \to Y$ is \textit{stratified} if
\begin{enumerate}
\item $f$ is \textit{stratum preserving}, i.e. for any stratum $S$ of $Y_k - Y_{k-1}$, $f^{-1}(S)$ is a union of strata of $X$. 
\item for each $p \in Y_k - Y_{k-1}$, there exists a neighborhood $N$ of $p$ in $Y_k$, a topologically stratified space 
$$F = F_k \supseteq F_{k-1} \supseteq \cdots \supseteq F_{-1} = \varnothing$$
and a strata preserving homeomorphism $F \times N \to f^{-1}(N)$ which commutes with projection to $N$.
\end{enumerate}
\end{defn}

\subsection{The Constructible Derived Category} 

Let $X$ be a topologically stratified space. A sheaf $\mathcal{L}$ on $X$ is \textit{locally constant} if for each $x \in X$, there exists an open set $U \subseteq X$ and an $R$-module $M$ such that $\mathcal{L}|_U \simeq M_U$, where $M_U$ is the constant sheaf on $U$ associated with the $R$-module $M$. A locally constant sheaf $\mathcal{L}$ with finitely generated stalks is referred to as a \textit{local system}. A complex of sheaves $S$ is \textit{cohomologically locally constant} (CLC) if the associated cohomology sheaves are locally constant. Now, let $\mathfrak{X}$ be any filtration of $X$ by closed subsets, not necessarily a stratification. A complex of sheaves $S$ is \textit{cohomologically locally constant with respect to $\mathfrak{X}$} ($\mathfrak{X}$-clc) if for each $k$, $S|_{X_k - X_{k-1}}$ is CLC. A complex of sheaves $S$ is \textit{constructible with respect to $\mathfrak{X}$} ($\mathfrak{X}$-cc) if $S$ is $\mathfrak{X}$-clc and the stalks of the cohomology sheaves are finitely generated. A complex of sheaves $S$ is \textit{topologically constructible} if $S$ is bounded and $S$ is constructible with respect to some stratification of $X$. In this paper, the word \textit{constructible} means topologically constructible. Let $D^b_c(X)$ denote the full subcategory of $D^b(X)$ consisting of constructible complexes and $D^b_{\mathfrak{X}}(X)$ denote the full subcategory of $D^b(X)$ consisting of $\mathfrak{X}$-cc complexes. The standard $t$-structure on $D^b(X)$ induces a $t$-structure on $D^b_c(X)$. The truncation functors are denoted $\tau_{\leq i}: D^b_c(X) \to D^{b, \leq i}_c(X)$ and $\tau_{\geq i}: D^b_c(X) \to D^{b, \geq i}_c(X)$. 

Useful references for sheaf theory are \cite{I,KS}. A brief discussion of the constructible derived category can be found in \cite[\S 1.3-\S 1.15]{GM2}. For a more complete discussion, we refer the reader to \cite{B}. We will record some of the most useful facts below for convenience.

Let $X$, $Y$ be stratified spaces with stratifications $\mathfrak{X}$ and $\mathfrak{Y}$ respectively. Let $f:X \to Y$ be a stratified map with respect to these stratifications. We have the four functors
\begin{center}
\begin{tikzcd}
D^b_{\mathfrak{X}}(X) \arrow[r, bend left, "{Rf_*, Rf_!}"] &D^b_\mathfrak{Y}(Y) \arrow[l, bend left, "{f^*, f^!}"].
\end{tikzcd}
\end{center}

\begin{prop}\label{top man}
If $X$ is an oriented manifold and $i:Z \to X$ is the inclusion of a locally closed oriented submanifold of codimension $d$, we have that $i^!R_X \simeq i^*R_X[-d]$.
\begin{proof}
See \cite[p. 336]{I}. 
\end{proof}
\end{prop}

There are adjunctions ($f^*,Rf_*$) and  ($Rf_!, f^!$). There is a morphism of functors $Rf_! \to Rf_*$ which is an isomorphism if $f$ is proper. For an open set $U \subseteq X$ and $Z = X - U$ its closed complement, we have inclusions 
$$U \xrightarrow{j}  X  \xleftarrow{i} Z. $$
Since $Z$ is closed, $Ri_! = i_!$. This gives rise to the \textit{adjunction distinguished triangles}
\begin{gather*}
i_!i^! \to id \to Rj_*j^* \xrightarrow{[1]},\\
Rj_!j^! \to id \to i_*i^* \xrightarrow{[1]}.
\end{gather*}

\begin{lem} \label{projection commute}
Let $M$ be a locally contractible topological space and $\pi : X' = X \times M \to X$ be the projection. Let $Y \subseteq X$ and $Y' = \pi^{-1}(Y)$. We have a cartesian diagram
\begin{center}
\begin{tikzcd}
Y' \arrow[r,"i'"] \arrow[d,"\pi'"] &X' \arrow[d, "\pi"]\\
Y \arrow[r,"i"] &X
\end{tikzcd}
\end{center}
\begin{enumerate}[(a)]
\item If $Y \subseteq X$ is open and $S \in D^b_c(Y)$, then 
\begin{equation*}
Ri'_*\pi'^{*}S \simeq \pi^*Ri_*S.
\end{equation*}
\item If $Y \subseteq X$ is closed and $T \in D^b_c(X)$, then
\begin{equation*}
\pi'^*i^!T \simeq i'^!\pi^*T.
\end{equation*}
\end{enumerate}
\begin{proof}
See \cite[V, 3.13]{B}. 
\end{proof}
\end{lem}

We end with the following important proposition.

\begin{prop} \label{morph lifting prop}
Suppose $A, B, C$ are objects in $\D^b_c(X)$ and $\H^a(A) = 0$ for $a \geq k+1$. Let $\psi : B \to C$ be a morphism such that the induced maps on cohomology $\H^a(B) \to \H^a(C)$ are isomorphisms for all $a \leq k$. Then the map induced by $\psi$
$$\text{Hom}_{D^b_c(X)}(A, B) \to \text{Hom}_{D^b_c(X)}(A, C)$$
is an isomorphism. 
\begin{proof}
See \cite[\S 1.15]{GM2}.
\end{proof}
\end{prop}

\section{Deligne's Construction for Complex Algebraic Varieties}\label{construction}

We briefly recall Deligne's construction when the complex algebraic variety $X$ has pure complex dimension $n$ with stratification $\mathfrak{X}$ by closed subvarieties. The stratification induces a filtration by open subsets
\begin{equation*}
U_1 \subseteq U_2 \subseteq \cdots \subseteq U_{n+1} = X
\end{equation*}
where $U_k = X-X_{n-k}$. Since $X$ is pure dimensional, $U_1$ is dense in $X$. Let $j_k :U_k \to U_{k+1}$ denote the inclusion maps. Define a complex recursively as follows: if $\L$ is a local system on the open dense union of strata $U_1$, then set 
\begin{equation*}
I_1 = \L[n]
\end{equation*}
\begin{equation*}
I_{k+1} = \tau_{\leq k-1-n} Rj_{k*}I_k
\end{equation*}
Note that here we are using the middle perversity and the indexing convention described at the end of the introduction.

We see that in the pure dimensional case, the starting point for Deligne's construction of the intersection complex is a local system on an open dense union of strata, shifted by the complex dimension of that open dense set. When the variety is not necessarily pure dimensional, the starting point for Deligne's construction will still be a local system on a open dense union of strata. However, the notion of shifting by dimension becomes more complicated. This is because an open dense set in the variety may consist of many components of different dimensions. Given a local system on an open dense set, restriction gives local systems on each component of fixed dimension. We can then shift each restricted local system by the dimension of the component that it is supported on. We make this more precise below.

In what follows, let $X$ be a complex algebraic variety of complex dimension $n$, with stratification 
$$\mathfrak{X} : X= X_n \supseteq X_{n-1} \supseteq \cdots \supseteq X_{-1}= \varnothing$$
so that all strata contained in $X_k - X_{k-1}$ are of pure complex dimension $k$ and $X$ has the structure of a topologically stratified space (e.g. $\mathfrak{X}$ is induced by a Whitney stratification). Let $p$ denote the middle perversity. Unless otherwise stated, the word dimension is taken to mean complex dimension. 

\begin{rmk}\label{alg var vs strat space}
In all of our proofs, we only use the fact that a complex algebraic variety $X$ has the structure of a topologically stratified space (in the sense of Definition \ref{top strat space def}). We do not use the algebraic structure of $X$ or of any of its strata. Thus, if one replaces the words complex algebraic variety by "topologically stratified space with only even dimensional strata" and is careful with the notion of dimension, then one obtains the same statements for this larger class of objects. In particular, our results will also hold for complex analytic spaces. If one is interested in the more general statement, then one should use the notion of topological dimension given in \cite{HW}. The main reason we make this simplification is to avoid constantly switching between complex dimension (more natural when stating our results) and real dimension (more natural when discussing stratified spaces). This will hopefully alleviate some of the confusion in the rest of the paper.
\end{rmk}

\subsection{Identifying the Open Dense Union of Strata}\label{U1}
In this section, we identify an open dense subset of the complex algebraic variety $X$ that will serve as the starting point of Deligne's construction. Fix a stratification $\mathfrak{X}$ of $X$. For each $0 \leq m \leq n$, let $U^m$ be the union of all $m$-dimensional strata which are open in $X$ and let $X^m \coloneqq \overline{U^m}$. Since the closure of a stratum is a union of strata of lower dimension by the axiom of the frontier, $X^m$ is a union of strata and $\partial X^m = X^m - U^m$ is a union of strata of lower dimension. In particular dim$_\CC\partial X^m \leq m-1$. Each $X^m$ is therefore a pseudomanifold with stratifications 
$$\mathfrak{X}^m : X^m_m \supseteq X^m_{m-1} \supseteq \cdots \supseteq X^m_0 \supseteq X^m_{-1} = \varnothing,$$ 
where $X^m_k = X^m \cap X_k$ and $X^m_k - X^m_{k-1}$ consists of strata of pure complex dimension $k$ for $k \leq m$. Set $U^m_k = X ^m - X_{m-k}$. Notice that in general, $U^m_k$ is only locally closed in $X$ and $U^m_1 = U^m$. Set
\begin{equation}
U_1 = \bigsqcup_{m=1}^n U^m.
\end{equation}
We will see in Corollary \ref{irred comp cor} that $X^m$ actually the union of all $m$-dimensional irreducible components of $X$.

\begin{rmk}
If $X$ is of pure dimension $n$, then $U^n = X - X_{n-1}$ and $U^m = \varnothing$ for $m < n$. Moreover, $U^n$ is dense in $X$ and $X^n = \overline{U^n} = X$. 
\end{rmk}

\begin{prop} \label{open dense strata}
The open set $U_1 = \bigsqcup_{m=1}^n U^m$ is dense, i.e. $\bigcup_{m=1}^n X^m = X$

\begin{proof}
Suppose $\bigcup_{m=1}^n X^m$ is strictly contained in $X$. Then the set complement $(\bigcup_{m=1}^n X^m)^c = \bigsqcup_{i \in I} S_i$ is a union of strata. Since the closure of any stratum is a union of lower dimensional strata, there are two cases. Fix any stratum $S_1 \subseteq (\bigcup_{m=1}^n X^m)^c $. 
\begin{case}
We have $S_1\subseteq \overline{S_k}$ for some $k \in I$. In this case, since $S_1^c = \bigcup_{m=1}^n X^m \sqcup \bigsqcup_{i \neq 1} S_i$, we have 
$$ \overline{S_1^c} = \bigcup_{m=1}^n X^m \cup \bigcup_{i \neq 1} \overline{S_i}.$$ 
Since $S_1 \subseteq \overline{S_k}$, we have that $S_1 \subseteq \overline{S_1^c} = \text{interior}(S_1)^c$. This implies that $S_1$ has empty interior which is a contradiction since $S_1$ is a nonempty stratum. 
\end{case}
\begin{case}
The strata $S_1$ does not meet $S_k$ for any $k \neq i$. This implies that 
$$ \overline{S_1^c} = \bigcup_{m=1}^n X^m \cup \bigcup_{i \neq 1} \overline{S_i} = \bigcup_{m=1}^n X^m \sqcup \bigsqcup_{i \neq 1} S_i = S_1^c.$$ 
It follows that $S_1^c = \overline{S_1^c} = \text{interior}(S_1)^c$, i.e. $S_1$ is open in $X$. This contradicts the definition of $X^m$. 
\end{case}
In either case, we have a contradiction. So we conclude that $(\bigcup_{m=1}^n X^m)^c = \varnothing$, i.e. $U_1$ is dense.
\end{proof}
\end{prop}

\subsection{The Open Filtration Induced by a Stratification}

In this section, we describe a filtration of $X$ by open subsets, beginning with $U_1$, induced by a stratification $\mathfrak{X}$. The following example shows that applying Deligne's construction to certain filtrations by open sets will not produce a direct sum of intersection complexes.

\begin{ex} \label{open filt ex}
Let $E \subseteq \PP^2$ be a smooth elliptic curve and $C_E \subseteq \CC^3$ be the affine cone over $E$. Let $L$ be a line in $\CC^3$ passing through the origin that is not contained in $C_E$ and $C' = C_E \cap \{z_3 = 1\} \subset \CC^3$. Let $X = C_E \cup L$. Consider the stratification
$$\mathfrak{X} : C_E \cup L \supset L \cup C' \supset \{0\} \supset \varnothing.$$
With the notation above, $U^2 = C_E - C' - \{0\}$ and $U^1 = L - \{0\}$. Taking closures, we have $X^2 = C_E$ and $X^1 = L$. We have sets
\begin{multicols}{2}
    \begin{enumerate}[itemsep=5pt]
        \item[] $U^2_1 = X^2 - X_1 = C_E - C' - \{0\}$,
        \item[] $U^2_2 = X^2 - X_0 = C_E - \{0\}$,
        \item[] $U^2_3 = X^2-X_{-1} = C_E$.
        \item[] $U^1_1 = X^1-X_0 = L-\{0\}$,                
        \item[] $U^1_2 = X^1-X_{-1} = L$,
    \end{enumerate}
    \end{multicols}
One possible way to filter $X$ by open subsets is the following. Let 
\begin{equation*}
\begin{aligned}
U_1 &= U^2_1 \cup U^1_1 =\left( C_E - C'- \{0\} \right) \cup \left(L-\{0\} \right),\\
U_2 &= U^2_2 \cup U^1_2 = \left(C_E - \{0\} \right) \cup L = X,\\
U_3 &= U^2_3 \cup U^1_3 = C_E \cup L = X.
\end{aligned}
\end{equation*}
This gives a filtration by open subsets 
$$U_1 \xrightarrow{j_1} U_2 \xrightarrow{j_2=id} U_3.$$
We apply Deligne's construction to this filtration. Recall that $p$ denotes the middle perversity. On the open dense set $U_1$, let $I_1 = \QQ_{U^2}[2] \oplus \QQ_{U^1}[1]$. On $U_2 = X$, if we truncate at $p(2)-2 = -2$, the complex appearing in Deligne's construction is 
$$\tau_{\leq p(2) - 2}Rj_{1*}I_1 = \tau_{\leq -2}Rj_{1*}\left ( \QQ_{U^2}[2] \oplus \QQ_{U^1}[1]\right) = \tau_{\leq -2}Rj_{1*}\QQ_{U^2}[2].$$
Here we see that the truncation operation kills off the contribution from the open 1-dimensional stratum. We add this contribution back in using Deligne's construction for $U^1$, i.e. on $U_2$, set 
$$I_2 = \tau_{\leq p(2)-2}Rj_{1*}I_1 \oplus \tau_{\leq p(2) - 1}Rj_{1*}\QQ_{U^1}[1] = \tau_{\leq -2}Rj_{1*}\QQ_{U^2}[2] \oplus \tau_{\leq - 1}Rj_{1*}\QQ_{U^1}[1].$$
Notice that $I_2$ is not a direct sum of intersection complexes since the first summand is truncated at $-2$ instead of $-1$. If we instead truncate at $p(4)-2 =-1$, we have on $U_2 = X$ the complex
$$I_2' = \tau_{\leq p(4) - 2}Rj_{1*}I_1 = \tau_{\leq -1}Rj_{1*}\QQ_{U^2}[2] \oplus \tau_{\leq - 1}Rj_{1*}\QQ_{U^1}[1].$$
However, the first summand of $I_2'$ is still not the intersection complex of $C_E$. The support condition fails for $\tau_{\leq -1}Rj_{1*}\QQ_{U^2}[2]$ since $\{x \in C_E \ | \ \H^1(\tau_{\leq -1}Rj_{1*}\QQ_{U^2}[2])_x \neq 0\} = C'$ is not zero dimensional.  
\end{ex}

The problem with the filtration in the example is that strata of differing dimensions were added at the same stage in the filtration. Our filtration of the complex algebraic variety $X$ by open sets described below avoids this issue and is motivated by the following observation. If $X$ is pure of dimension $n$ with stratification $\mathfrak{X}$, then the induced filtration by open subsets is given by
$$\varnothing \subseteq U_1 \subseteq \cdots \subseteq U_{n+1} = X,$$
where $U_k = X - X_{n-k}$. It follows that $U_{k+1}-U_k = (X-X_{n-k-1}) - (X-X_{n-k}) = X_{n-k}-X_{n-k-1}$
consists of all codimension $k$ strata of $X$. None of these strata can be open since any open subset of pure dimensional variety $X$ has dimension $n$. So $U_{k+1} - U_k$ consists of all non-open codimension $k$ strata of $X$. We would like our filtration of $X$ by open sets to satisfy the same property.

Let 
\begin{gather}
W_k = \bigcup_{m=n-k+2}^n U^m_{m-n+k},\\
U_k = W_k \sqcup \bigsqcup_{m=1}^{n-k+1}U^m_1.
\end{gather}

A priori, the sets $U_k$ are not necessarily open in $X$ since the sets $U^m_{m-n+k} = X^m - X_{n-k}$ are only locally closed in $X$. However, we have the following lemma.

\begin{lem}
The set $U_k$ is open in $U_{k+1}$ for each $ 1 \leq k \leq n$.
\begin{proof}
We show that if $p \in U_k$, there is a neighborhood $N$ of $p$ in $U_{k+1}$ that is contained in $U_k$. If $p \in \bigsqcup_{m=1}^{n-k} U^m$, then we are done. If $p \in \bigcup_{m=n-k+1}^n U^m_{m-n+k}$, let $N = U_{k+1} \cap X_{n-k}^c$. Since $\bigsqcup_{m=1}^{n-k} U^m \subseteq X_{n-k}$, we see that $N = W_{k+1} \cap X^c_{n-k}$. Notice that $p \in N$ and $N$ is open in $U_{k+1}$. We claim that $N \subseteq U_k$. Let $q \in N$. Since $q \in W_{k+1}$, $q \in U^m_{m-n+k+1} = X^m -X_{n-k-1}$ for some $m \geq n-k+1$. Since $q \in X_{n-k}^c$, we see that $q \in U^m_{n-k} \subseteq U_k$. 
\end{proof}
\end{lem}

Since $U_{n+1} = X$, the previous lemma implies that $U_n$ is open in $X$. It follows from descending induction on $k$ that $U_k$ is open in $X$ for all $1 \leq k \leq n$. This gives a finite filtration $\mathfrak{U}$ of $X$ by open subsets
\begin{equation}\label{open filt eq}
\mathfrak{U}: \varnothing \subseteq U_1 \subseteq \cdots \subseteq U_n \subseteq U_{n+1} = X.
\end{equation}
We have inclusions $U_k \xrightarrow{j_k} U_{k+1} \xleftarrow{i_k} \left( U_{k+1} - U_k\right)$. We will refer to the filtration $\mathfrak{U}$ as the \textit{open filtration induced by} $\mathfrak{X}$. 

We conclude this section with several facts about the structure of the open filtration $\mathfrak{U}$. 

\begin{lem}\label{Wk+1-Wk}
We have $U_{k+1} - U_k = \left( W_{k+1} - W_k \right) - U^{n-k+1}_1$.
\begin{proof}
Notice that
\begin{equation*}
\begin{aligned}
U_{k+1} - U_k &= U_{k+1} - \left(W_k \sqcup \bigsqcup_{m=1}^{n-k+1} U^m_1\right)\\
&= \left( U_{k+1} - W_k \right) - \bigsqcup_{m=1}^{n-k+1} U^m_1\\
&= \left (\left(W_{k+1}-W_k\right) - \bigsqcup_{m=1}^{n-k+1} U^m_1\right) \cup \left( \left( \bigsqcup_{m=1}^{n-k} U^m_1 - W_k \right) - \bigsqcup_{m=1}^{n-k+1} U^m_1 \right)\\
&= \left (W_{k+1}-W_k \right) - \bigsqcup_{m=1}^{n-k+1} U^m_1\\
&= \left( W_{k+1}-W_k \right) - U^{n-k+1}_1,
\end{aligned}
\end{equation*}
where the last equality holds since $\bigsqcup_{m=1}^{n-k} U^m_1 \subseteq W_{k+1}^c$. 
\end{proof}
\end{lem}

\begin{lem} \label{non open strat lem}
The set $U_{k+1} - U_k$ consists of all non-open $(n-k)$-dimensional strata, i.e. $X_{n-k} - X_{n-k-1} = \left(U_{k+1} - U_k\right) \sqcup U^{n-k}_1$.
\begin{proof}
Suppose $x \in X_{n-k}-X_{n-k-1}$. Let $S \subseteq X_{n-k}-X_{n-k-1}$ be the $(n-k)$-dimensional stratum containing $x$. Since $X^m$ is a union of strata and $X = \cup_{m=1}^nX^m$, $S \subseteq X^m$ for some $m \geq n-k$. Since $S \subseteq X_{n-k-1}^c$, $S \subseteq X^m - X_{n-k-1} \subseteq U^m_{m-n+k+1}\subseteq U_{k+1}$. If $S$ is open, then $S \subseteq U^{n-k}_1$. If $S$ is not open, then $S \subseteq W_{k+1}$. In this case, suppose that $S \subseteq U_k$. Since $S$ is not open, $S \subseteq W_k$. In particular $S \subseteq U^m_{m-n+k} = X^m - X_{n-k}$ for some $m \geq n-k+2$. This implies that $S \subseteq X_{n-k}^c$ which is a contradiction. So $x \in S \subseteq U_{k+1}-U_k$. It follows that $X_{n-k} - X_{n-k-1} \subseteq  \left(U_{k+1} - U_k\right) \sqcup U^{n-k}_1$.

Conversely, if $x \in U_{k+1}-U_k$, then $x \notin U_k$ implies that $x \notin U^m_{m-n+k}= X^m - X_{n-k}$ for all $m$. It follows that $x \notin X_{n-k}^c$, i.e. $x \in X_{n-k}$. Since $x \in U_{k+1}$, $x \in U^m_{m-n+k+1} = X^m - X_{n-k-1}$ for some $m$. In particular, $x \in X_{n-k-1}^c$. It follows that $x \in X_{n-k}-X_{n-k-1}$.  If $x \in U^{n-k}$, then $x \in X_{n-k}-X_{n-k-1}$ by definition. It follows that $\left(U_{k+1} - U_k\right) \sqcup U^{n-k}_1 \subseteq X_{n-k} - X_{n-k-1}$. 
\end{proof}
\end{lem}

\begin{lem} \label{U^m closed in U_k}
Fix $1 \leq k \leq n$. Then $U^m_{m-n+k}$ is closed in $U_k$ for $n-k+1 \leq m \leq n$ and $U^m_1$ is closed in $U_k$ for $1 \leq m \leq n-k$. 
\begin{proof}
Suppose $n-k+1 \leq m \leq n$. Since $X^m$ is closed in $X$, it suffices to show that $U^m_{m-n+k} = X^m \cap U_k$. The inclusion $U^m_{m-n+k} \subseteq X^m \cap U_k$ follows from the definition of $U_k$. Now let $x \in X^m \cap U_k$. Since $x \in X^m$ and $m \geq n-k+1$, $x \in W_k$ or $x \in U^{n-k+1}_1$. It follows that $x \in U^l_{l-n-k} = X^l - X_{n-k}$ for some $n-k+1 \leq l \leq n$. In particular, $x \notin X_{n-k}$. It follows that $x \in X^m - X_{n-k} = U^m_{m-n+k}$. We conclude that $U^m_{m-n+k} = X^m \cap U_k$. 

A similar argument shows that $U^m_1 = X^m \cap U_k$ for $1 \leq k \leq n-k$.
\end{proof}
\end{lem}

\subsection{Construction of $IC(\mathfrak{X},\L)$}

Let $\mathfrak{X}$ be a stratification of $X$ and $\mathfrak{U}$ the open filtration induced by $\mathfrak{X}$. Let $\L$ be a local system on the open dense subset $U_1 \subseteq X$. 

\begin{rmk}\label{local sys open dense rmk}
We can express $\L$ as $\L = \bigoplus_{m=1}^n a^m_{1*}\L^m$ where $\L^m \coloneqq \L|_{U^m}$ is a local system on $U^m$ and $a^m_1:U^m \to U_1$ is inclusion of a closed subset. We will often abuse notation and identify $a^m_{1*}\L^m$ with $\L^m$. Since each $\L^m$ is a local system on $U^m$, we can associate with $\L$ the complex $\bigoplus_{m=1}^n \L^m[m]$.
\end{rmk}

Define a complex $IC(\mathfrak{X}, \L)$ on $X$ recursively as follows: set

\begin{equation}\label{gen Deligne cons}
\begin{gathered}
I_1 = \bigoplus_{m=1}^n \L^m[m] \text{ on } U_1,\\
I_{k+1} = \tau_{\leq k-1-n}Rj_{k*}I_k \oplus \bigoplus_{m=1}^{n-k} \L^m[m] \text{ on } U_{k+1},
\end{gathered}
\end{equation}
and let $IC(\mathfrak{X}, \L) = I_{n+1}$. Note that the truncation is done with respect to the middle perversity.

We refer to $IC(\mathfrak{X},\L)$ as the object obtained by the Deligne's construction with respect to the stratification $\mathfrak{X}$ and the local system $\L$. Note that this construction only uses the filtration structure of $\mathfrak{X}$. We emphasize that we are only shifting the local system by the complex dimension of $U^m$. This shift is done so that our complex $IC(\mathfrak{X},\L)$ agrees with the indexing convention discussed at the end of the introduction.

We show below that the complex $IC(\mathfrak{X}, \L)$ can be interpreted as a direct sum of Deligne-Goresky-MacPherson intersection complexes. 
Let $X$ be a complex algebraic variety of complex dimension $n$ with stratification $\mathfrak{X}$. Recall that the stratification $\mathfrak{X}$ induces an open dense subset $U_1 = \bigsqcup_{m=1}^n U^m$. We saw that $X = \bigcup_{m=1}^n X^m$ where $X^m = \overline{U^m}$. Corollary \ref{irred comp cor} will imply that $X^m$ can also be interpreted as the union of all $m$-dimensional irreducible components of $X$.
Let $\L = \bigoplus_{m=1}^n\L^m$ be a local system on the open dense union of strata $U_1$. Let $IC(\mathfrak{X}^m,\L^m)$ be the object obtained by Deligne's construction with respect to the induced stratification $\mathfrak{X}^m$ of $X^m$ and the local system $\L^m$ on $U^m$ for the pure dimensional variety $X^m$. Notice that $IC(\mathfrak{X}^m,\L^m)$ is precisely the Deligne-Goresky-MacPherson intersection complex of $X^m$. Let $a^m: X^m \to X$ be inclusion.

\begin{prop} \label{IC is direct sum}
With the notation above, we have that $$IC(\mathfrak{X},\L) \simeq \bigoplus_{m=1}^n a^m_* IC(\mathfrak{X}^m, \L^m).$$ 
\begin{proof}
Fix $1 \leq k \leq n$ and $n-k+1 \leq m \leq n$. Let $\bullet \coloneqq m-n+k$. Consider the cartesian diagram
\begin{center}
\begin{tikzcd}
U_{\bullet}^m\arrow[r, "a^m_\bullet"] \arrow[d, "j^m_\bullet"] \arrow[dr, phantom, "\square"] & U_{k} \arrow[d,"j_k"] \\
U_{\bullet+1}^m \arrow[r,"a^m_{\bullet +1}"] & U_{k+1}
\end{tikzcd}
\end{center}
 where all maps are inclusions.  Lemma \ref{U^m closed in U_k} implies that the maps $a^m_\bullet$ and $a^m_{\bullet + 1}$ are inclusions of closed subsets. It follows that
\begin{equation*}
Rj_{k*}a^m_{\bullet*}  \simeq R(j_k \circ a^m_\bullet)_* = R(a^m_{\bullet +1} \circ j^m_{\bullet})_* \simeq a^m_{\bullet+1*} Rj^m_{\bullet*}.
\end{equation*}
Now, notice that the complex $IC(\mathfrak{X},\L)$ is a direct summand of complexes of the form
\begin{equation*}
\begin{aligned}
\tau_{\leq -1}Rj_{n*} \cdots \tau_{\leq -m} Rj_{n-m+1*} a^m_{1*} \L^m[m].
\end{aligned}
\end{equation*}
Using the above commutation relation and the fact that $a^m_{\bullet*}$ is exact, we can iteratively move $a^m_{1*}$ to the left. We conclude that
\begin{equation*}
\tau_{\leq -1}Rj_{n*} \cdots \tau_{\leq -m} Rj_{n-m+1*} a^m_{1*} \L^m[m] \simeq a^m_* \tau_{\leq -1} Rj^m_{m*} \cdots \tau_{\leq -m} Rj^m_{1*}\L^m[m].
\end{equation*}
It follows that $IC(\mathfrak{X},\L) = \bigoplus_{m=1}^n a^m_* IC(\mathfrak{X}^m, \L^m)$. 
\end{proof}
\end{prop}

\begin{rmk}\label{indexing conv rmk}
With our choice of indexing convention in the construction of $IC(\mathfrak{X},\L)$, each summand appearing in $Rj_{k*}I_k$ is truncated at the same place. If one uses the Borel convention and does not shift the initial local systems when constructing the complex $IC(\mathfrak{X}, \L)$, the summands appearing in $Rj_{k*}I_k$ will need to be truncated in different places to ensure that $IC(\mathfrak{X}, \L)$ is a direct sum of the middle perversity Deligne-Goresky-MacPherson intersection complexes. It is therefore simpler notationally to use our indexing convention when describing the construction of $IC(\mathfrak{X},L)$ than the Borel convention. Additionally, notice that with our indexing convention, the cohomology sheaves of $IC(\mathfrak{X},L)|_{W_{k+1}}$ vanish above degrees $k-1-n$ by definition. A crucial point is that this vanishing condition can be stated without using the fact that $IC(\mathfrak{X},\L)$ is a direct sum of complexes and we use it in our axiomatic characterization of $IC(\mathfrak{X},\L)$ (see Axioms [AX$1'$] in Definition \ref{AX$1'$}). One can go from the Borel indexing convention to our indexing convention (and vice-versa) by shifting each summand by the appropriate complex dimension. 
\end{rmk}

\begin{rmk}\label{arb perv rmk}
If one wishes to consider other perversities, one can try to mimic the above construction of $IC(\mathfrak{X},\L)$ (in this context, the Borel convention seems more natural). However, there is not a clear analogue of the middle perversity vanishing conditions mentioned in the previous remark. Due to this, characterizing the intersection complex for arbitrary perversities seems like a more subtle question. 
\end{rmk}

We conclude this section by illustrating the construction in the setting of Example \ref{open filt ex}.  
\begin{ex}
With the same notation as Example \ref{open filt ex}, the open filtration $\mathfrak{U}$ induced by the stratification $\mathfrak{X}$ is given by
$$\mathfrak{U}: U_1 \xrightarrow{j_1} U_2 \xrightarrow{j_2} U_3 = X,$$
where
\begin{equation*}
\begin{aligned}
U_1 &= U^2_1 \cup U^1_1 =\left( C_E - C'- \{0\} \right) \cup \left(L-\{0\} \right),\\
U_2 &= U^2_2 \cup U^1_1 = \left(C_E - \{0\} \right) \cup L -\{0\},\\
U_3 &= U^2_3 \cup U^1_2 = C_E \cup L = X.
\end{aligned}
\end{equation*}
Deligne's construction proceeds as follows. On $U_1$, set $I_1 = \QQ_{U^2}[2] \oplus \QQ_{U^1}[1]$. On $U_2$, set 
$$I_2 = \tau_{\leq -2}Rj_{1*}I_1 \oplus \QQ_{U^1}[1] = \tau_{\leq-2}Rj_{1*} \QQ_{U^2}[2] \oplus \QQ_{U^1}[1].$$
On $U_3$, set
\begin{gather*}
I_3 = \tau_{\leq -1} Rj_{2*}I_2 = \tau_{\leq -1} Rj_{2*} \left(\tau_{\leq-2}Rj_{1*} \QQ_{U^2}[2] \oplus \QQ_{U^1}[1]\right) \\
= \tau_{\leq -1} Rj_{2*}\tau_{\leq-2}Rj_{1*} \QQ_{U^2}[2] \oplus \tau_{\leq -1}Rj_{2*} \QQ_{U^1}[1].
\end{gather*}
Here we see that both summands of $IC(\mathfrak{X}) = I_3$ are intersection complexes. 
\end{ex}

\section{An Axiomatic Characterization of $IC(\mathfrak{X},\L)$}\label{axiomatic char}
When the complex algebraic variety is pure dimensional, Goresky and MacPherson give a stratification independent set of axioms characterizing the intersection complex in \cite{GM2}. We recall the axioms with respect to the middle perversity for pure dimensional varieties below for convenience.

\begin{defn} \label{AX2}
Let $X$ be a complex algebraic variety of pure complex dimension $n$. A topologically constructible complex $S$ satisfies axioms [AX2] if
\begin{enumerate}[(a)]
\item (Normalization) $S|_{X-\Sigma} = \L[n]$ where $\Sigma \subset X$ is a closed subset of complex dimension $n-1$ and $\L$ is a local system on $X-\Sigma$,
\item (Lower Bound) $\H^a(S) = 0$ for $a < -n$,
\item (Support) dim$_\CC\{x \in X \ | \ \H^a(i_x^*S)\neq 0 \} < -a$ for $a > -n$,
\item (Cosupport) dim$_\CC\{x \in X \ | \ \H^a(i_x^!S)\neq 0 \} < a$ for $a < n$,
\end{enumerate}
where $i_x:\{x\} \to X$ is inclusion. These axioms differ slightly from the ones proposed by Goresky and MacPherson in \cite{GM2} because we normalize using complex dimension rather than real dimension. 
\end{defn}

Let $X$ be a possibly reducible complex algebraic variety of complex dimension $n$. Let $X^m$ be the union of all $m$-dimensional irreducible components of $X$. Then each $X^m$ is a variety of pure dimension $m$. Let $IC(X^m)$ be the corresponding intersection complexes (with $\QQ$ coefficients). Recall that the intersection complex (with $\QQ$ coefficients) $IC(X)$ of $X$ is defined to be $IC(X) = \bigoplus_{m=1}^n IC(X^m)$. Since each summand satisfies the support and cosupport axioms, one might guess that the direct summand satisfies the support and cosupport axioms. The next example shows that this is not the case. One might also guess that the complex $IC(X)|_{X^m}$ satisfies axioms [AX2] since each summand satisfies axioms [AX2]. If this were true, there would be a natural map  $IC(X) \to \bigoplus_{m=1}^n IC(X^m)$ via the adjunction maps. The next example shows that this is also not the case.

\begin{ex} \label{ax 2 fails}
Inside $\CC^3$, let $P = \{(z_1,z_2,0) | z_i \in \CC\}$ and $L = \{(0,0,z_3) | z_3 \in \CC\}$. Let $X = P \cup L$ be the reducible variety with irreducible components $P$ and $L$. The intersection complex of $X$ is given by $IC = IC(P) \oplus IC(L) = \QQ_P[2] \oplus \QQ_L[1]$. The support and cosupport axioms [AX2](c)(d) fail for $IC$ since
\begin{equation*}
\text{dim}_\CC\{x \in X \ | \ \H^{-1}(i_x^*IC) \neq 0\} = \text{dim}_\CC L = 1 \neq 0,
\end{equation*}
and 
\begin{equation*}
\text{dim}_\CC\{x \in X \ | \ \H^{1}(i_x^!IC) \neq 0\} = \text{dim}_\CC L = 1 \neq 0,
\end{equation*}
where $i_x:\{x\}\to X$ is inclusion.
If we instead consider $IC|_P = \QQ_P[2] \oplus \tilde{i}_{0*}\QQ[1]$ where $\tilde{i}_0:\{0\} \to P$ is the inclusion, the support condition axiom [AX2](c) is satisfied. However, notice that
\begin{equation*}
\tilde{i}_0^!(IC|_P) = \tilde{i}_0^! \QQ_P[2] \oplus \tilde{i}_0^!\tilde{i}_{0*}\QQ[1] = \QQ[-2] \oplus \QQ[1].
\end{equation*}
This implies that the cosupport condition [AX2](d) fails for $IC|_P$ since $\{x \in P \ | \ \H^{-1}(\tilde{i}_x^!IC|_P) \neq 0\} = \{0\} \neq \varnothing$.
\end{ex}

In the previous example, we see that the cosupport axiom fails because we first restrict the complex $IC$ to the irreducible component $P$. If we do not first restrict, notice that 
$$i^!_0(IC) = i^!_0\QQ_P[2] \oplus i^!_0\QQ_L[1] = \QQ[-2] \oplus \QQ[-1].$$
This implies that 
$$\text{dim}_\CC\{x \in P \ | \ \H^1(i_x^!IC)\neq 0\} = \text{dim}_\CC\{0\} = 0.$$
We conclude that for $a < 2$, $\text{dim}_\CC\{x \in P \ | \ \H^a(i_x^!IC)\neq 0\} < a$. The significance of this observation is that although neither $IC$ nor $IC|_P$ satisfies the cosupport condition, $IC$ satisfies a \textit{pure dimensional} analog of the cosupport condition. We will show in the following sections that a pure dimensional analog of the support and cosupport axioms will help us characterize the complex $IC$. 

In the following sections, let $X$ be a complex algebraic variety of complex dimension $n$ with stratification $\mathfrak{X}$. Consider the open filtration
$$\mathfrak{U}: U_1 \subseteq \cdots \subseteq U_n \subseteq U_{n+1} = X,$$ 
 induced by the stratification $\mathfrak{X}$ as in Equation \ref{open filt eq}. Recall that $U_1 = \bigsqcup_{m=1}^n U^m$ where $U^m$ is the union of all open $m$-dimensional strata in $X$. Let $\L$ be a local system on $U_1$. As in Remark \ref{local sys open dense rmk}, we write $\L = \bigoplus_{m=1}^n\L^m$ where $\L^m$ is a local system on $U^m$ extended to $U_1$ by zero.

\subsection{Axioms [AX1$'$]}

\begin{defn} \label{AX$1'$}
Let $S$ be a complex on $X$ and set $S_{k} \coloneqq S|_{U_k}$. We have inclusions $U_k \xrightarrow{j_k} U_{k+1} \xleftarrow{i_k} \left( U_{k+1} - U_k \right)$. Recall that $W_k = \bigcup_{m=n-k+2}^n U^m_{m-n+k}$ where $U^m_{m-n+k}=X^m - X_{n-k}$. We say that $S$ satisfies axioms [AX1$'$] (with respect to the stratification $\mathfrak{X}$) if
\begin{enumerate}[(a)]
\item (Normalization) $S|_{U_1} \simeq \bigoplus_{m=1}^n \L^m[m]$ in $D^b_c(U_1)$, 
\item (Vanishing) for all $k \geq 1$, $\H^a(S|_{W_{k+1}}) = 0$ for $a > k-1-n$,
\item (Attaching) the induced morphism on cohomology sheaves $$\H^a(i_k^*S_{k+1}) \to \H^a(i_k^*Rj_{k*}j_k^*S_{k+1})$$ is an isomorphism for all $k \geq 1$ and $a \leq k-1-n$.
\end{enumerate}
\end{defn}

\begin{rmk}\label{ax1' vs ax1}
The stratification dependent axioms [AX1$'$] are analogous to the stratification dependent axioms [AX1] for pseudomanifolds proposed by Goresky and MacPherson in \cite{GM2}. When $X$ is a pseudomanifold, axioms [AX1$'$] reduce to axioms [AX1]. One difference between the axioms is that we do not include a lower bound axiom. This is because the lower bound axiom for pseudomanifolds is actually implied by the other axioms (in particular [AX1](a) and (d)) and is not needed to characterize the intersection complex. We will also not need an analog of the lower bound axiom to characterize the complex $IC(\mathfrak{X},\L)$. The normalization axiom [AX1$'$](a) differs from [AX1](a) in that our open dense set $U_1$ contains strata of differing dimensions. We require that each local system is shifted based on the dimension of the strata that it is supported on. The vanishing axiom [AX1$'$](b) differs from [AX1](c) in that we restrict our complex $S$ to the smaller open set $W_{k+1}$ instead of $U_{k+1}$. The reason for this is that the open set $U_{k+1}$ contains the open strata $U^m$ for $n-k \leq m \leq n$. The normalization axiom implies that $S|_{U^m} \simeq \L^m[m]$. We must therefore ignore these strata if we want the vanishing axiom to hold. The attaching axiom [AX1$'$](c) is completely analogous to [AX1](d). They both give the same vanishing condition for the cohomology sheaves when restricting the complex to the non-open $(n-k)$-dimensional strata. 

We also do not require that the complex $S$ is $\mathfrak{X}$-cc. We will eventually see that if $S$ satisfies axioms [AX1$'$], then $S$ is $\mathfrak{X}$-cc. This is analogous to Borel's discussion of constructibility in the pseudomanifold case; see \cite[V, \S 3]{B}.
\end{rmk}

\subsection{Alternative Formulations of [AX1$'$](c)} 

In this section, we give two useful alternative characterizations of [AX1$'$](c), namely [AX1$'$](c$'$) and [AX1$'$](c$''$). Recall the adjunction distinguished triangle
$$i_{k!}i_k^!S_{k+1} \to S_{k+1} \to Rj_{k*}j_k^*S_{k+1} \xrightarrow{[1]}.$$

Restricting gives the distinguished triangle
$$i_k^!S_{k+1} \to i_k^*S_{k+1} \to i_kRj_{k*}j_k^*S_{k+1} \xrightarrow{[1]}.$$

The long exact sequence in cohomology and [AX1$'$](c) imply that $\H^a(i_k^!S_{k+1}) = 0$ for $a \leq k-n$. So we see that [AX1$'$](c) is equivalent to 

\begin{itemize}
\item[(c$'$)] $\H^a(i_k^!S_{k+1}) = 0$ for $k \geq 1$ and $a \leq n - k$. 
\end{itemize}

We now relate this to the vanishing of the costalks $\H^a(i_x^!S)$. Fix $k \geq 1$. Suppose $x \in U_{k+1} - U_k$. Factor the inclusion $i_x: \{x\} \to X$ into 

\begin{center}
\begin{tikzcd}
\{x\} \arrow[r, "i_x"] \arrow[d, "\mu_x"] & X \\
U_{k+1} - U_k \arrow[r, "i_k"] & U_k \arrow[u, "\alpha"]
\end{tikzcd}
\end{center}

It follows that 
\begin{equation*}
\begin{aligned}
i_x^!S &= \mu_x^! \circ i_k^! \circ \alpha^!S\\
&= \mu_x^! \circ i_k^! S_{k+1}\\
&= \mu_x^* \circ i_k^!S_{k+1}[-2(n-k)],
\end{aligned}
\end{equation*}
where the second equality holds because $\alpha$ is an open inclusion and the third equality follows from Proposition \ref{top man} since $U_{k+1}-U_k$ is a topological manifold of real dimension $2(n-k)$. It follows that $$\H^a(i^!_xS) = \H^{a-2(n-k)}(S_{k+1})_x.$$ Hence we see that [AX1$'$](c') is equivalent to 

\begin{itemize}
\item[(c$''$)] If $x \in U_{k+1} - U_k$, then $\H^a(i_x^!S) = 0$ for all $a \leq n-k$. 
\end{itemize}

\subsection{[AX1$'$] Characterizes $IC(\mathfrak{X},\L)$}
The main goal of this section is to prove the following theorem. 
\begin{thm} \label{main thm 2}
Let $X$ be a complex algebraic variety of complex dimension $n$ with stratification $\mathfrak{X}$. Let $\mathfrak{U}$ be the open filtration induced by $\mathfrak{X}$. Let $\L = \bigoplus_{m=1}^n \L^m$ be any local system on the open dense union of strata $U_1 \subseteq X$. The functor $F$ which takes the complex $\bigoplus_{m=1}^n \L^m[m]$ to the complex $IC(\mathfrak{X},\L)$ defines an equivalence of categories between 
\begin{enumerate}[(a)]
\item the full subcategory of $D^b_c(U_1)$ whose objects are all complexes of the form $\bigoplus_{m=1}^n \L^m[m]$ where $\L^m$ is a local system on $U^m_1$ extended to $U_1$ by zero, and
\item the full subcategory of $D^b_c(X)$ whose objects are all complexes satisfying axioms [AX1$'$]
\end{enumerate}
The inverse functor $G$ assigns to any complex $S$ satisfying axioms [AX1$'$] the complex \newline $\bigoplus_{m=1}^n \H^{-m}(S|_{U_1})[m]$.
\end{thm}

We have the two immediate corollaries.
\begin{cor} \label{main cor 1}
If a complex $S$ satisfies [AX1$'$], then $S$ is canonically isomorphic to \\ 
$F(\L)=IC(\mathfrak{X},\L)$ in $\D^b_c(X)$.  
\end{cor}

\begin{cor}
If a complex $S$ satisfies [AX1$'$], then $S$ is $\mathfrak{X}$-cc.
\begin{proof}
Since $S$ satisfies [AX1$'$], $S$ is isomorphic to $IC(\mathfrak{X},\L)$. Since $IC(\mathfrak{X},\L)$ is constructed by iterated pushforwards along strata and truncations applied to the constructible complex $\bigoplus_{m=1}^n \L^m[m]$, it is constructible. Therefore, $S$ is $\mathfrak{X}$-cc.
\end{proof}
\end{cor}

To prove Theorem \ref{main thm 2}, we make the following reduction. For each $k \geq 1$, let $\C_k$ denote the full subcategory of $D^b_c(U_k)$ consisting of complexes which satisfy axiom [AX1$'$] on $U_k$. If $S \in \C_k$, then $S$ is a complex on 
$$U_k = W_k \sqcup \bigsqcup_{m=1}^{n-k+1}U^m_1.$$ 
Notice that $W_k$ is closed in $U_k$ and let $i^W_k:W_k \to U_k$ be the inclusion. The normalization and vanishing axioms imply that $S|_{U^m_1} \simeq \H^{-m}(S)$ is a local system. We set $\L^m \coloneqq \H^{-m}(S)$ for $1 \leq m \leq n-k+1$. Since $U_k$ is a disjoint union of $W_k$ and the $U^m_1$'s, $S$ can be expressed as 
$$S = S_{W_k} \oplus \bigoplus_{m=1}^{n-k+1} \L^m[m],$$
where $S_{W_k} = i^W_{k*}i^{W*}_k S$. We will denote the adjunction map $S \to S_{W_k}$ by $pr_1$ and the direct sum of adjunction maps $S \to \bigoplus_{m=1}^{n-k+1} \L^m[m]$ by $pr_2$. 

For any $S \in \C_k$, define $F_k(S)$ by:
\begin{equation}
F_k(S) = \tau_{\leq k-1-n}Rj_{k*}S \oplus \bigoplus_{m=1}^{n-k} \H^{-m}(S)[m].
\end{equation}
We claim that $F_k$ is a functor from $\C_k$ to $\C_{k+1}$. It suffices to show that for any $S \in \C_k$, $F_k(S) \in \C_{k+1}$. The normalization and vanishing axioms are all satisfied by definition of $F_k(S)$. Since
\begin{equation}\label{jkFk=id}
\begin{aligned}
j_k^*F_k(S) &= j_k^*\left(\tau_{\leq k-1-n}Rj_{k*}S \oplus \bigoplus_{m=1}^{n-k} \H^{-m}(S)[m]\right)\\
&= \tau_{\leq k-1-n}S \oplus \bigoplus_{m=1}^{n-k}\L^m[m]\\
&= \tau_{\leq k-1-n} \left(S_{W_k} \oplus \bigoplus_{m=1}^{n-k+1} \L^m[m] \right) \oplus \bigoplus_{m=1}^{n-k}\L^m[m]\\
&= S_{W_k} \oplus \L^{n-k+1}[n-k+1] \oplus \bigoplus_{m=1}^{n-k}\L^m[m]\\
&= S,
\end{aligned}
\end{equation}
the attaching axiom is satisfied because the attaching morphism is the composition
$$\tau_{\leq k-1-n}i_k^*Rj_{k*}S \simeq i_k^*F_k(S) \to i_k^*Rj_{k*}j_k^*F_k(S) \simeq i_k^*Rj_{k*}S.$$ The restriction functor $j_k^*$ is clearly a functor from $\C_{k+1}$ to $\C_k$. 

The key observation is that our original functor $F$ is the composition $F = F_n \circ F_{n-1} \circ \cdots \circ F_1$ and the inverse functor $G$ is the composition $G=j_1^* \circ \cdots \circ j_{n-1}^* \circ j_n^*$. Theorem \ref{main thm 2} is therefore a consequence of the following theorem.  

\begin{thm}
For $k \geq 1$, the functor $F_k$ defines an equivalence of categories between $\C_k$ and $\C_{k+1}$. The inverse functor $G_k$ is $j_k^*$.
\begin{proof}
Equation \ref{jkFk=id} shows that $j_k^*F_k = id_{\C_k}$ as a functor. We must also show that $F_kj_k^*$ is isomorphic to $id_{\C_{k+1}}$ as functors, i.e. for any $S \in \C_{k+1}$, we must construct an isomorphism $S \to F_kj_k^*S$ such that for any morphism $S \to T$ in the category $\C_{k+1}$, the diagram

\begin{center}
\begin{tikzcd}
S \arrow[r] \arrow[d]
& T \arrow[d] \\
F_kj_k^*(S) \arrow[r]
& F_kj_k^*(T)
\end{tikzcd}
\end{center}
commutes. We construct the morphism $S \to F_kj_k^*S$ as follows. Since $S \in \C_{k+1}$, $S = S_{W_{k+1}} \oplus \bigoplus_{m=1}^{n-k} \L^m[m]$. It follows that 

\begin{equation*}
\begin{aligned}
F_kj_k^*S &= \tau_{\leq k-1-n}Rj_{k*}j_k^*(S_{W_{k+1}} \oplus \bigoplus_{m=1}^{n-k} \L^m[m])  \oplus \bigoplus_{m=1}^{n-k} \L^m[m]\\
&= \tau_{\leq k-1-n}Rj_{k*}j_k^*S_{W_{k+1}}  \oplus \bigoplus_{m=1}^{n-k} \L^m[m].
\end{aligned}
\end{equation*}

The adjunction morphism gives us a morphism $S_{W_{k+1}} \to Rj_{k*}j_k^*S_{W_{k+1}}$. The vanishing axiom [AX1$'$](b) implies that $S_{W_{k+1}} \simeq \tau_{\leq k-1-n}S_{W_{k+1}}$. Therefore, we have a morphism 
$$S \xrightarrow{pr_1}S_{W_{k+1}} \simeq \tau_{\leq k-1-n}S_{W_{k+1}} \to \tau_{\leq k - 1 -n}Rj_{k*}j_k^*S_{W_{k+1}} \simeq \tau_{\leq k - 1 -n} Rj_{k*}j_k^*S.$$
We also have a morphism 
$$S \xrightarrow{pr_2} \bigoplus_{m=1}^{n-k} \L^m[m].$$
Taking the direct sum of these morphisms gives us a morphism $S \to F_kj_k^*S$.
By construction, the morphism $S \to F_kj_k^*S$ is an isomorphism over $U_k$. We need to check that it is an isomorphism over $U_{k+1} - U_k$. The attaching axiom [AX1$'$](c) implies that $$i_k^*S \to i_k^*Rj_{k*}j_k^*S$$ induces an isomorphism on cohomology sheaves for all  $a \leq k-1-n$. Thus, 
$$ i_k^*S \simeq \tau_{\leq k-1-n} i_k^*S \simeq \tau_{\leq k-1-n} i_k^*Rj_{k*}j_k^*S \simeq i_k^*F_kj_k^*S.$$

We have thus constructed an isomorphism $S \to F_kj_k^*S$. Since this morphism is constructed as a direct sum of two morphisms, we will check that each summand is a morphism of functors. Let $f:S \to T$ be a morphism in the category $\C_{k+1}$. Consider the diagram
\begin{center}
\begin{tikzcd}
S  \arrow[d, "pr_1"] \arrow[rrr, "f"] & & & T\arrow[d, "pr_1"] \\
S_{W_{k+1}}  \arrow[rrr, "f_{W_{k+1}}"] \arrow[dd, "\eta_{k+1}(S)"] \arrow[dr] & & & T_{W_{k+1}} \arrow[dd, "\eta_{k+1}(T)"] \arrow[dl] \\
& Rj_{k*}j_k^*S_{W_{k+1}} \arrow[r] & Rj_{k*}j_k^*T_{W_{k+1}}\\
\tau_{\leq k-1-n}Rj_{k*}j_k^*S_{W_{k+1}} \arrow[rrr,"g_{W_{k+1}}"] \arrow[ur] & & & \tau_{\leq k-1-n}Rj_{k*}j_k^*T_{W_{k+1}} \arrow[ul, "\theta"]
\end{tikzcd}
\end{center}
where $g_{W_{k+1}} = \tau_{\leq k - 1 -n} Rj_{k*}j_k^*(f_{W_{k+1}})$. It is clear that the top square commutes and the two trapezoids commute. The left and right triangles commute by the truncation distinguished triangle. The commutativity of the top and bottom trapezoids combined with the commutativity of the left and right triangles imply that 
$$\theta \circ \eta_{k+1}(T) \circ f_{W_{k+1}} = \theta \circ g_{W_{k+1}} \circ \eta_{k+1}(S).$$
Since $S_{W_{k+1}} \simeq \tau_{\leq k-1-n}S_{W_{k+1}}$ and $\theta$ induces isomorphisms on cohomology sheaves for all $a \leq k-1-n$, Proposition \ref{morph lifting prop} implies that 
$$\eta_{k+1}(T) \circ f_{W_{k+1}} = g_{W_{k+1}} \circ \eta_{k+1}(S).$$
It follows that the bottom rectangle commutes. Commutativity of the upper and lower rectangles implies that the largest rectangle commutes. Since the diagram
\begin{center}
\begin{tikzcd}
S \arrow[d,"pr_2"] \arrow[r, "f"] &T \arrow[d,"pr_2"]\\
\bigoplus_{m=1}^{n-k} \L_S^m[m] \arrow[r] &\bigoplus_{m=1}^{n-k} \L_T^m[m]
\end{tikzcd}
\end{center}
commutes, we conclude that the isomorphism $id_{\C_{k+1}} \to F_kj_k^*$ is an isomorphism of functors.
\end{proof}
\end{thm}

\subsection{Axioms [AX2$'$]}

In this section, we give a stratification independent collection of axioms characterizing $IC(\mathfrak{X},\L)$. Let $X$ be a complex algebraic variety of complex dimension $n$.

\begin{defn} \label{AX$2'$}
Suppose that $S$ is $\mathfrak{X}$-clc for some stratification $\mathfrak{X}$ of $X$. We say that $S$ satisfies axioms [AX2$'$] if
\begin{enumerate}[(a)]
\item (Normalization) There exists an open dense subset $V$ of $X$ such that $V = \bigsqcup_{m=1}^n V^m$ where $V^m$ is a topological manifold of complex dimension $m$, dim$_\CC(\overline{V^m} - V^m) \leq m-1$, and there exist local systems $\L^m$ on $V^m$ such that $S|_{V^m} \simeq \L^m[m]$
\item (Pure Dimensional Support) For $1 \leq m \leq n$, if $a > -m$, $$\text{dim}_\CC\{x \in \overline{V^m} \ | \ \H^a(i_x^*S) \neq 0\} < -a.$$ 
\item (Pure Dimensional Cosupport) For $1 \leq m \leq n$, if $a < m$, $$\text{dim}_\CC\{x \in \overline{V^m} \ | \ \H^a(i_x^!S) \neq 0\} < a.$$ 
\end{enumerate}
where $i_x : \{x\} \to X$ is the inclusion.
\end{defn}

\begin{rmk}\label{ax2' vs ax2}
The stratification independent axioms [AX2$'$] are analogous to axioms [AX2] proposed by Goresky and MacPherson in \cite{GM2}; see Definition \ref{AX2} for axioms [AX2]. When $X$ is a pure dimensional complex algebraic variety, axioms [AX2$'$] reduce to axioms [AX2]. Again, we do not include an analog of the lower bound axiom because it is not needed to characterize the complex $IC(\mathfrak{X},\L)$. The normalization axiom [AX2$'$](a) differs from [AX2](a) in that the open dense set $V$ contains manifolds of differing dimensions. We require that each local system is shifted by the complex dimension of the manifold. The pure dimensional support axiom [AX2$'$](b) differs from [AX2](b) in a significant way. Instead of looking at all possible stalks of the complex $S$, we look at stalks of $S$ in a specific $V^m$. For each $m$, we place a condition on the vanishing of cohomology of these stalks in certain degrees. The specific degrees subject to our conditions depend on $m$ instead of the dimension of the complex algebraic variety. The difference between [AX2$'$](c) and [AX2](c) is similar to the difference between [AX2$'$](b) and [AX2](b).  

We also make a remark on the assumption that $S$ is $\mathfrak{X}$-clc. In \cite{GM2}, it is assumed that $S$ is topologically constructible, i.e. the cohomology sheaves of $S$ also have finitely generated stalks. The finite generation of the stalks of the cohomology sheaves is a consequence of the axioms by Corollary \ref{main cor 1} and the following proposition. 
\end{rmk}

\begin{prop} \label{main prop}
Let $X$ be a complex algebraic variety of complex dimension $n$ and let $\mathfrak{X}$ be a stratification of $X$ by closed subvarieties. Suppose that $S$ is $\mathfrak{X}$-clc. Then $S$ satisfies [AX1$'$] with respect to $\mathfrak{X}$ if and only if $S$ satisfies [AX2$'$]. 
\end{prop}

Before proving the proposition, we will need to establish several lemmas. Let $\mathfrak{X}$ be a topological stratification of $X$. Recall that $U^m$ is the union of all open $m$-dimensional strata of $X$ and $X^m$ is defined to be $\overline{U^m}$. Let $W^m$ be the largest set of points in $X$ which admit a neighborhood homeomorphic to $\CC^{m}$. We can equivalently think of $W^m$ as the largest open subset of $X$ which is a topological manifold of complex dimension $m$.

\begin{lem} \label{strat contain}
With the notation above, we have $U^m \subseteq W^m$.
\begin{proof}
Let $p \in U^m$. Let $S^m_p \subseteq X_m - X_{m-1}$ be the open complex $m$-dimensional stratum containing $p$. By definition of topologically stratified space, there exists a neighborhood $N_p$ and a real $(2(n-m)-1)$-dimensional topologically stratified space $L$ such that $N_p \simeq \CC^m \times cone^o(L)$. Recall from the definition of stratified space that the stratification of $L$ induces one on $N_p$. Since $S^m_p$ is open, we can take $N_p \subseteq S^m_p \subseteq X_m$. It follows that $N_p = N_p \cap X_m \simeq \CC^m \times cone^o(L_{-1}) = \CC^m$. So $p \in W^m$.
\end{proof}
\end{lem}

\begin{lem}
Suppose that $V$ is any open dense subset of $X$ consisting of points in $X$ which admit a neighborhood homeomorphic to some $\CC^{m}$. Write $V = \bigsqcup_{m=1}^m V^m$ where $V^m$ is a topological manifold of complex dimension $m$. Then $\overline{V^m} = X^m$.
\begin{proof}
We will show that $\overline{V^m}$ and $X^m$ are both equal to $\overline{W^m}$. We first show that $X^m = \overline{W^m}$. Lemma \ref{strat contain} implies that $X^m \subseteq \overline{W^m}$. We now show that $W^m \subseteq X^m$. Let $p \in W^m$ and suppose that $p \notin X^m$. Since $p \in W^m$, there exists a distinguished neighborhood $N_p$ of $p$ homeomorphic to $\CC^m$. Since $X = \bigcup_{l=1}^n X^l$, $p \in X^l$ for some $l \neq m$. Since $X^l = \overline{U^l}$, $N_p \cap U^l$ must be nonempty. Let $q \in N_p \cap U^l$. Since $q \in N_p$, $q$ admits a neighborhood homeomorphic to $\CC^m$. Since $q \in U^l$, Lemma \ref{strat contain} implies that $q$ admits a neighborhood homeomorphic to $\CC^l$. This is a contradiction because $l \neq m$. It follows that $p \in X^m$.
\par
The proof that $\overline{V^m} = \overline{W^m}$ is similar.
\end{proof}
\end{lem}

\begin{cor}\label{irred comp cor}
Let $X$ be a complex algebraic variety of complex dimension $n$ with stratification $\mathfrak{X}$. Let $U^m$ be the union of all open $m$-dimensional strata. Then $\overline{U^m}$ is the union of all $m$-dimensional irreducible components of $X$. 
\begin{proof}
To see this, let $\tilde{X}^m$ be the union of all $m$-dimensional irreducible components of $X$. Let $V^m$ be the smooth locus of $\tilde{X}^m - \left( \bigcup_{l \neq m} \tilde{X}^l \cap \tilde{X}^m \right)$. Then $V = \bigsqcup_{m=1}^n V^m$ is an open dense subset of $X$ consisting of points which admit a neighborhood homeomorphic to $V^m$. It then follows from the previous lemma that $\tilde{X^m} = \overline{U^m}$.
\end{proof}
\end{cor}

\begin{lem} \label{4a=5a}
Let $S$ be an $\mathfrak{X}$-clc complex. Then $S$ satisfies [AX1$'$](a) if and only if $S$ satisfies [AX2$'$](a) . 
\begin{proof}
If $S$ satisfies [AX1$'$](a) with respect to $\mathfrak{X}$, then the open set $U_1$ coming from the stratification also satisfies the requirements in [AX2$'$](a). Now let $S$ be $\mathfrak{X}$-clc and suppose $S$ satisfies [AX2$'$](a). Since $S$ is $\mathfrak{X}$-clc, $S|_{U^m}$ is CLC. In particular, all of  the cohomology sheaves $\H^a(S)|_{U^m}$ are locally constant. For $a \neq -m$, [AX2$'$](a) implies that $\H^a(S)|_{U^m \cap V^m} = 0$. Since $\H^a(S)|_{U^m}$ is locally constant and its restriction to $U^m \cap V^m$ is $0$, we conclude that $\H^a(S)|_{U^m} = 0$. This proves the lemma. 
\end{proof}
\end{lem}

\begin{rmk} \label{local sys rmk}
Let $S$ be a $\mathfrak{X}$-clc complex and suppose $S$ satisfies [AX2$'$](a). Then $S|_V^m = \L^m[m]$ where $\L^m$ is a local system on the topological manifold $V^m$. By the previous lemma, the assumption that $S$ is $\mathfrak{X}$-clc implies that $S|_{U^m}\simeq \L'^m$ where $U^m$ is the open subset of $X^m$ coming from the stratification and $\L'^{m}$ is a local system on $U^m$. Since dim$_\CC(\overline{V^m}-V^m) \leq m-1$, $U^m \cap V^m$ has real codimension greater than or equal to $2$ in $U^m$. This implies that there is a surjection of fundamental groups $\pi_1(U^m \cap V^m) \twoheadrightarrow \pi_1(U^m)$. Fix a base point $x \in U^m \cap V^m$. The local system $\L'$ on $U_1$ corresponds to a representation $\phi : \pi_1(U_1, x) \to Aut(\L'_x)$ and the restriction $\L|_{U_1 \cap V}$ corresponds to a representation $\tilde{\phi}: \pi_1(U^m \cap V^m, x) \to Aut(\L^m_x)$. Since $\L'_x = S_x = \L_x$, we have a commutative diagram:
\begin{center}
\begin{tikzcd}
\pi_1(U^m \cap V^m, x) \arrow[d, twoheadrightarrow, "i_*"] \arrow[dr,"\tilde{\phi}"]\\
\pi_1(U^m, x) \arrow[r,"\phi"] &Aut(\L_x)
\end{tikzcd}
\end{center}
Surjectivity of the fundamental groups implies that $\phi$ is the unique representation making this diagram commute. To see this, let $\psi$ be another such representation. Then for any $[\gamma] \in \pi_1(U^m,x)$, surjectivity of the fundamental groups says there exists  $[\sigma] \in \pi_1(U^m \cap V^m, x)$ such that $i_*([\sigma]) = [\gamma]$. It follows that 
$$\phi([\gamma]) = \tilde{\phi}([\sigma]) = \psi([\gamma]),$$
and so $\psi = \phi$. This implies that the local system $\L'^m$ on $U^m$ is the unique extension of the local system $\L^m|_{U^m \cap V^m}$. The most important case of this is the constant sheaf. If $\L^m|_{U^m \cap V^m} \simeq R_{U^m \cap V^m}$, then the representation $\tilde{\phi}$ is trivial. Surjectivity of the fundamental groups implies that the representation $\phi$ is trivial, i.e. $\L^m|_{U^m} \simeq R_{U^m}$.

This fact is false without the surjectivity of fundamental groups. Consider the inclusion $S^1-\{p\} \to S^1$ and take any nontrivial local system on $S^1$. Then its restriction to $S^1-\{p\}$ is trivial. 
\end{rmk}

\begin{rmk} \label{union of strat}
The significance of Lemma \ref{4a=5a} is the following. If $S$ is $\mathfrak{X}$-clc and satisfies [AX2$'$], then we can replace the open set $V$ appearing in [AX2$'$](a) with the open set $U_1$ coming from the stratification. Since $S$ is $\mathfrak{X}$-clc, the sets appearing in [AX2$'$](b) and [AX2$'$](c) can be taken to be unions of strata.
\end{rmk}

We are now ready to prove Proposition \ref{main prop}.
\begin{proof}[Proof of Proposition \ref{main prop}]
Suppose $S$ is an $\mathfrak{X}$-clc complex and that $S$ satisfies [AX2$'$]. Lemma \ref{4a=5a} implies that $S$ satisfies [AX1$'$](a). We now prove that $S$ satisfies [AX1$'$](b) if and only if $S$ satisfies [AX2$'$](b). Fix $1 \leq m \leq n$ and $a > -m$. By Remark \ref{union of strat}, the set $\{x \in X^m \ | \ \H^a(i_x^*S) \neq 0\}$ is a union of strata. Suppose $S$ satisfies [AX1$'$](b). This implies that the strata contained in $\{x \in X^m \ | \ \H^a(i_x^*S) \neq 0\}$ cannot meet $W_{k+1}$ for $ a > k - 1 -n$. Hence they can only be contained in $W_{k+1}$ for $a \leq k - 1 - n$, equivalently $k \geq a + n +1$. So the strata are contained in $W_{k+1} - W_k$ for some $k \geq a + n + 1$. By Lemma \ref{Wk+1-Wk}, $U_{k+1} - U_k = (W_{k+1} - W_k) - U^{n-k+1}$. Since $k \geq a + n + 1$ and $a > -m$, we have that $n-k+1 \leq -a < m$. It follows that $U^{n-k+1}$ cannot be among the strata contained in $\{x \in X^m \ | \ \H^a(i_x^*S) \neq 0\}$. This implies that the only allowable strata are contained in $U_{k+1} - U_k$ for $n-k < -a$. It follows that dim$_\CC\{x \in X^m \ | \ \H^a(i_x^*S) \neq 0\} \leq n-k < -a$.

Conversely, suppose $S$ satisfies [AX2$'$](b). Then dim$_\CC\{x \in X^m \ | \ \H^a(i_x^*S) \neq 0\} < -a$. Since $\{x \in X^m \ | \ \H^a(i_x^*S) \neq 0\}$ is a union of strata, it can only contain strata of dimension $< -a$. These strata are contained in $U_{k+1}-U_k \subseteq W_{k+1} - W_k$ for $n-k < -a$. So these strata can only be contained in $W_{k+1}$ for $n-k+1 \geq -a$ or equivalently, $a \leq k - 1 -n$. This implies that $S|_{W_{k+1}} \simeq \tau_{\leq k - 1 -n} S|_{W_{k+1}}$. 

We now prove that $S$ satisfies [AX1$'$](c) if and only if it satisfies [AX1$'$](c). Fix $1 \leq m \leq n$ and $a < m$. Again by Remark \ref{union of strat}, the set $\{x \in X^m \ | \ \H^a(i_x^!(S) \neq 0\}$ is a union of strata. If $x \in U^m$, then by factoring the inclusion $i_x: \{x\} \to X$ as 
\begin{center}
\begin{tikzcd}
\{x\} \arrow[r, "i_x"] \arrow[d, "\mu_x"]&X\\
U^m \arrow[ru, "j^m" ']
\end{tikzcd}
\end{center}
we see that $i^!S = \mu_x^! S|_{U^m_1} =\mu_x^*\L^m[-m]$. Since $a < m$, we see that $\{x \in X^m \ | \ \H^a(i_x^!(S) \neq 0\}$ does not contain any open $m$-dimensional strata. 

Now, suppose $S$ satisfies [AX2$'$](c), then $S$ also satisfies [AX1$'$](c$''$). In particular, this implies that these strata cannot meet $U_{k+1} - U_k$ for $a \leq n-k$. Thus the only allowable strata are contained in $U_{k+1} - U_k$ for $a > n-k$. It follows that dim$_\CC\{x \in X^m \ | \ \H^a(i_x^!(S) \neq 0\} \leq n-k < a$. 

Conversely suppose $S$ satisfies [AX2$'$](c). Then dim$_\CC\{x \in X^m \ | \ \H^a(i_x^!(S) \neq 0\} < a$. Since $\{x \in X^m \ | \ \H^a(i_x^!S) \neq 0\}$ is a union of non-open strata, it can only contain strata of complex dimension $< a$. These strata are contained in $U_{k+1} - U_k$ for $a > n-k$. 
\end{proof}

\section{Topological Independence of $IC(\mathfrak{X},\L)$}\label{top indep}
The main goal of this section is to prove Theorem \ref{main thm 1}. 

\begin{thm'}
Let $X$ be a complex algebraic variety of complex dimension $n$ which is not necessarily pure dimensional. Let $U$ be an open dense subset of $X$ such that $U = \bigsqcup_{m=1}^n U^m$ where $U^m$ is a topological manifold of complex dimension $m$ and dim$_\CC(\overline{U^m}-U^m) \leq m-1$. Let $\L^m$ be a local system on $U^m$ and set $\L = \bigoplus_{m=1}^n \L^m$ (extend each $\L^m$ on $U^m$ to $U$ by zero). Then there exists a unique (up to canonical isomorphism) complex $IC(X, \mathcal{L})$ satisfying axioms [AX2$'$], i.e.  $IC(X, \mathcal{L})$ is the unique complex satisfying: 
\begin{enumerate}[(a)]
\item  (Normalization) There exists an open dense subset $V$ of $X$ such that $V = \bigsqcup_{m=1}^n V^m$ where $V^m$ is a topological manifold of complex dimension $m$, dim$_\CC(\overline{V^m}-V^m) \leq m-1$, and $IC(X,\L)|_{V^m} \simeq \L'^m[m]$ where $\L'$ is the unique extension of $\L^m|_{U^m \cap V^m}$ to $V^m$ (see Remark \ref{local sys rmk} for more details on $\L'$).
\item (Pure Dimensional Support) For $1 \leq m \leq n$, if $a > -m$, $$\text{dim}_\CC\{x \in \overline{V^m} \ | \ \H^a(i_x^*S) \neq 0\} < -a.$$ 
\item (Pure Dimensional Cosupport) For $1 \leq m \leq n$, if $a < m$, $$\text{dim}_\CC\{x \in \overline{V^m} \ | \ \H^a(i_x^!S) \neq 0\} < a.$$ 
\end{enumerate}
\end{thm'}

To prove Theorem \ref{main thm 1}, we follow the same strategy as Goresky and MacPherson in \cite{GM2}. The main difficulty is that we need some way of comparing objects in $D^b_c(X)$ satisfying [AX1$'$] with respect to two different stratifications, which may not have a common refinement. To address this, we will construct a canonical filtration $\mathfrak{X}^{can}$ such that:
\begin{enumerate}
\item each topological stratification is a refinement of $\mathfrak{X}^{can}$,
\item applying Deligne's construction with respect to $\mathfrak{X}^{can}$ yields a complex $J^{can}$ satisfying [AX2$'$],
\item $J^{can}$ is $\mathfrak{X}$-clc for any stratification $\mathfrak{X}$.
\end{enumerate}

The existence of such a complex $J^{can}$ implies Theorem \ref{main thm 1} as follows. Suppose $S$ is $\mathfrak{X}$-clc for some stratification $\mathfrak{X}$ of $X$ and $S$ satisfies [AX2$'$]. Then $S$ satisfies [AX1$'$] with respect to $\mathfrak{X}$ by Proposition \ref{main prop}. Similarly, the complex $J^{can}$ described above also satisfies [AX1$'$] with respect to $\mathfrak{X}$. By Corollary \ref{main cor 1}, $S$ and $J^{can}$ are canonically isomorphic in $D^b_c(X)$. If $T$ is the complex obtained by applying Deligne's construction to any other stratification $\mathfrak{\tilde{X}}$, then $T$ satisfies [AX1$'$] with respect to $\mathfrak{\tilde{X}}$ and satisfies [AX2$'$] by Proposition \ref{main prop}. It follows that $T$ is also canonically isomorphic to $J^{can}$. 

\subsection{Construction of the Canonical Filtration}  
We will construct the canonical filtration $\mathfrak{X}^{can}$ inductively. For each $1 \leq m \leq n$, let $W^m$ be the largest set of points in $X$ which admit a neighborhood homeomorphic to $\CC^m$ and let $X^m = \overline{W^m}$. Set $X^{can}_{n-1} = X - W^n$. Now, suppose that
\begin{equation*}
\mathfrak{X}^{can}_k: X^{can}_{n} \supseteq X^{can}_{n-1} \supseteq \cdots \supseteq X^{can}_{n-k}
\end{equation*}
has been defined and each $X^{can}_{n-l}$ is closed in $X$. Recall that the open filtration $\mathfrak{U}^{can}_k$ induced by $\mathfrak{X}^{can}_k$ is given by 
\begin{equation*}
\mathfrak{U}^{can}_k: U^{can}_1 \subseteq \cdots \subseteq U^{can}_k,
\end{equation*}
where
\begin{equation*}
U^{can}_l = \left(X^n -  X^{can}_{n-l} \right) \cup \left(X^{n-1} - X^{can}_{n-l+1}\right) \cup \cdots \cup \left(X^{n-l+2} - X^{can}_{n-2}\right) \sqcup \bigsqcup_{m=1}^{n-l+1}W^m.
\end{equation*}
Let $J^{can}_k \in D^b_c(U^{can}_k)$ be the complex obtained by applying Deligne's construction with respect to the filtration $\mathfrak{X}^{can}_k$. Let $h_k \colon U^{can}_k \to X$ be the open inclusion. Let $V'$ be the largest open subset of $X^{can}_{n-k} - W^{n-k}$ which is a topological manifold of complex dimension $n-k$ and such that $\left(Rh_{k*}J^{can}_k\right)|_{X^{can}_{n-k}}$ is CLC. Let $V = V' \sqcup W^{n-k}$ and define $X^{can}_{n-k-1} \coloneqq X^{can}_{n-k} - V$. Notice that $X^{can}_{n-k-1}$ is closed in $X$.

\begin{lem}
$X^{can}_{n-1}$ is a union of strata for any stratification $\mathfrak{X}$.
\begin{proof}
Fix a stratification $\mathfrak{X}$. Recall that $X^{can}_1 = X - W^n$. Since $X$ is a union of strata, it suffices to show that $W^n$ is a union of strata. We claim that $W^n$ is a union of the strata $S_r$ which in the normal direction, look like $\CC^{n-r}$. If $x$ is contained in such a stratum, then $x$ has a neighborhood homeomorphic to $\CC^m$. Conversely, if $x$ has a neighborhood homeomorphic to $\CC^m$ and $S_r$ is the stratum containing $x$, then by possibly shrinking the neighborhood, we see that $S_r$ must look like $\CC^{n-r}$ in the normal direction.
\end{proof}
\end{lem}

\begin{prop} 
For $0 \leq k \leq n$, we have
\begin{enumerate}
\item For any stratification $\mathfrak{X}$, $X^{can}_{n-k-1}$ is a union of strata,
\item dim$_\CC X^{can}_{n-k-1} \leq n-k-1$,
\item $Z^{can}_{n-k} = (X^{can}_{n-k} - X^{can}_{n-k-1})- W^{n-k}$ is either empty or a $n-k$ complex dimensional topological manifold.
\item Let $J^{can}$ be the object obtained by applying Deligne's construction with respect to the canonical filtration $\mathfrak{X}^{can}$ and a local system $\L$ on $\bigsqcup_{m=1}^n W^m$. Then $J^{can}|_{Z^{can}_{n-k}}$ is CLC. 
\end{enumerate}
\begin{proof}
We prove (1)-(4) by induction on $k$. If $k = 0$, then $X^{can}_{n-1}$ is a union of strata by the previous lemma. Moreover, $W^n$ contains all of the $n$-dimensional strata of $X$ by Lemma \ref{strat contain}, so dim$_\CC X^{can}_{n-1} \leq n-1$. This shows that $(1)$ and $(2)$ are satisfied. Since $Z^{can}_{n} = (X - X^{can}_{n-1}) - W^n = \varnothing$ is empty, $(3)$ and $(4)$ are also satisfied. 

Now fix $k > 0$ and suppose that (1)-(4) hold for all integers strictly less than $k$. Induction hypothesis (1) says that $X^{can}_{n-k}$ is a union of strata. If we can show that the set $V$ used to define $X^{can}_{n-k-1}$ is a union of strata which contains the $n-k$ complex dimensional strata of $X$, then (1) and (2) will hold for $k$. Property (3) will hold for $k$ since
$$Z_{n-k} = (X^{can}_{n-k} - X^{can}_{n-k-1})- W^{n-k} = V - W^{n-k} = V'$$ 
and $V'$ is a topological manifold of complex dimension $n-k$. Finally, since $Rh_{k*}J^{can}_k$ is CLC on $Z^{can}_{n-k}$, 
$$J^{can}|_{Z^{can}_{n-k}} = \tau_{\leq k-1-n}\left(Rh_{k*}h_k^*J^{can}\right)|_{Z^{can}_{n-k}} = \tau_{\leq k-1-n}\left(Rh_{k*}J^{can}_k\right)|_{Z^{can}_{n-k}}$$
is CLC. This implies that (4) will hold for $k$. Thus it suffices to show that $V$ is a union of strata which contains the $n-k$ complex dimensional strata of $X$. This is a consequence of the following lemma.
\end{proof}
\end{prop}

\begin{lem}
In the situation above, the complex $Rh_{k*}J^{can}_k|_{X^{can}_{n-k}}$ is $\mathfrak{X}$-clc for $k \geq 1$. 
\begin{proof}
Denote the stratification of $X$ 
$$\mathfrak{X} : X = X_n \supseteq X_{n-1} \supseteq \cdots \supseteq X_{-1} = \varnothing.$$
By the induction hypothesis, $X^{can}_{n-k}$ is a union of strata. We must show that $Rh_{k*}J^{can}_k|_{X^{can}_{n-k}}$ is CLC on each stratum. Let $x \in X^{can}_{n-k}$ and let $S_r \subseteq X_{2r} - X_{2r-1}$ be the stratum containing $x$. By definition of topologically stratified space, there exists a distinguished neighborhood $N$ and a $(2(n-r) - 1)$ real dimensional topologically stratified space $L$ such that $N \simeq \CC^r \times cone^o(L)$ and $N \cap X_{2r+l'+1} \simeq \CC^r  \times cone^o(L_{l'})$.  Let $V \coloneqq cone^o(L)$ and $\pi : \CC^r  \times V \to V$ be projection onto the second factor. For $l \leq k$, let 
$$\tilde{U}^{can}_l \coloneqq U^{can}_l \cap N,$$ 
and 
$$\hat{U}^{can}_l \coloneqq \pi(U^{can}_l).$$
Let $\tilde{j}_l: \tilde{U}^{can}_l \to \tilde{U}^{can}_{l+1}$ and $\hat{j}_l: \hat{U}^{can}_l \to \hat{U}^{can}_{l+1}$ denote inclusions. For $l \leq k$, the induction hypothesis (1) ensures that $U^{can}_l$ is a union of strata. Remark \ref{projection rmk} implies that $\pi^{-1}(\pi(\tilde{U}^{can}_l))$. It follows that
\begin{multline*}
(Rh_{k*}J^{can}_k)|_N \simeq R\tilde{h}_{k*} \bigg(\tau_{\leq k-1-n} R\tilde{j}_{k-1*}\cdots \tau_{\leq -n}R\tilde{j}_{1*}\pi^*\hat{\L}^n[n] \\
\oplus \cdots \oplus \tau_{\leq k -1 -n} R\tilde{j}_{k-1*}\pi^*\hat{\L}^{n-k-2}[n-k-2] \oplus \bigoplus_{m=1}^{n-k+1} \pi^*\hat{\L}^m[m]\bigg). 
\end{multline*}
By Lemma \ref{projection commute}, moving $\pi^*$ to the left changes tildes to hats. This gives
\begin{multline*}
(Rh_{k*}J^{can}_k)|_N \simeq \pi^*R\hat{h}_{k*} \bigg(\tau_{\leq k-1-n} R\hat{j}_{k-1*}\cdots \tau_{\leq -n}R\hat{j}_{1*}\hat{\L}^n[n] \\
\oplus \cdots \oplus \tau_{\leq k -1 -n} R\hat{j}_{k-1*}\hat{\L}^{n-k-2}[n-k-2] \oplus \bigoplus_{m=1}^{n-k+1} \hat{\L}^m[m]\bigg). 
\end{multline*}
Since $V_{0}$ is a point, the complex $(Rh_{k*}J^{can}_k)|_{\pi^{-1}(V_{0})}$ is CLC. 
\end{proof}
\end{lem}

\begin{prop}
Let $J^{can}$ be the complex obtained from Deligne's construction with respect to the canonical filtration $\mathfrak{X}^{can}$ and some local system $\L$ on $\bigsqcup_{m=1}^n W^m$. Then $J^{can}$ satisfies [AX2$'$]. 
\begin{proof}
$J^{can}$ satisfies [AX2$'$](a) by construction. To verify [AX2$'$](b), fix $1 \leq m \leq n$ and $a > -m$. We want to show that dim$_\CC\{x \in X^m \ | \ \H^a(i_x^*J^{can}) \neq 0 \} < -a$. First, notice that $W^l \cap X^m$ is nonempty if and only if $l = m$. Since $J^{can}|_{W^m} \simeq \L^m[m]$ and $a > -m$, the set $W^l \cap \{x \in X^m \ | \ \H^a(i_x^*J^{can}) \neq 0 \}$ is empty for all $1 \leq l \leq n$. Thus, it suffices to consider the intersection 
$$Z^{can}_{n-k} \cap \{x \in X^m \ | \ \H^a(i_x^*J^{can}) \neq 0 \}.$$
Since $J^{can}|_{Z^{can}_{n-k}} \simeq \tau_{\leq k - 1 -n} J^{can}|_{Z^{can}_{n-k}}$, this intersection is possibly nonempty if and only if $a \leq k - 1 - n$.  Since dim$_\CC Z^{can}_{n-k} \leq n-k$, it follows that 
$$\text{dim}_\CC \left(Z^{can}_{n-k} \cap  \{x \in X^m \ | \ \H^a(i_x^*J^{can}) \neq 0 \}\right) \leq n-k < -a.$$ 
Since this is true for any $k \geq 1$, we conclude that dim$_\CC\{x \in X^m \ | \ \H^a(i_x^*J^{can}) \neq 0 \} < -a$. 

To verify [AX2$'$](c), fix $1 \leq m \leq n$ and $a < m$. A similar argument to the above shows that $W^l \cap \{x \in X^m \ | \ \H^a(i_x^!J^{can}) \neq 0 \}$ is empty for all $1 \leq l \leq n$. Again, it suffices to consider $Z^{can}_{n-k} \cap \{x \in X^m \ | \ \H^a(i_x^!J^{can}) \neq 0 \}$. Notice that by Proposition \ref{non open strat lem}, $Z^{can}_{n-k} = U^{can}_{k+1} - U^{can}_k$. The inclusions $U^{can}_{k} \xrightarrow{j} U^{can}_{k+1} \xleftarrow{i} U^{can}_{k+1}-U^{can}_k$ give rise to the adjunction triangle
$$i_!i^!J^{can}|_{U^{can}_{k+1}} \to J^{can}|_{U^{can}_{k+1}} \to Rj_*j^* J^{can}|_{U^{can}_{k+1}} \xrightarrow{[1]}.$$
Restriction to $U^{can}_{k+1}-U^{can}_k$ gives 
$$i^!J^{can}|_{U^{can}_{k+1}} \to i^*J^{can}|_{U^{can}_{k+1}} \to i^*Rj_*j^* J^{can}|_{U^{can}_{k+1}} \xrightarrow{[1]}.$$
Since $i^*J^{can}|_{U^{can}_{k+1}} \simeq \tau_{\leq k - 1 -n} i^*Rj_*j^*J^{can}|_{U^{can}_{k+1}}$ by construction, the long exact sequence in cohomology implies that $\H^a(i^!J^{can}|_{U^{can}_{k+1}}) = 0$ for $a \leq k-n$. Factor the inclusion $i_x : \{x\} \to X$ into 
\begin{center}
\begin{tikzcd}
\{x\} \arrow[r,"i_x"] \arrow[d, "\mu_x"] & X\\
U^{can}_{k+1} - U^{can}_k \arrow[r,"i"] & U^{can}_{k+1} \arrow[u,"\beta"]
\end{tikzcd}
\end{center}
Since $U^{can}_{k+1} - U^{can}_k$ is a topological manifold of dimension $2(n-k)$, Proposition \ref{top man} implies that 
$$i_x^!J^{can} = \mu_x^!i^!J^{can}|_{U^{can}_{k+1}} = \mu_x^*i^!J^{can}|_{U^{can}_{k+1}}[-2(n-k)],$$
where the first equality holds since $U^{can}_{k+1}$ is open in $X$. It follows that $\H^a(i_x^!J^{can}) = 0$ for $a \leq n-k$. So $(U^{can}_{k+1}-U^{can}_k ) \cap  \{x \in X^m \ | \ \H^a(i_x^!J^{can}) \neq 0 \}$ is possibly nonempty if and only if $a > n-k$. We conclude that 
$$\text{dim}_\CC\left(Z^{can}_{n-k} \cap  \{x \in X^m \ | \ \H^a(i_x^!J^{can}) \neq 0 \} \right)\leq n-k < a.$$
Since this is true for any $k \geq 1$, we conclude that dim$\{x \in X^m \ | \ \H^a(i_x^!J^{can}) \neq 0 \} < a$. 
\end{proof}
\end{prop}

\end{document}